\documentclass[11pt,a4paper]{amsart}
\usepackage[colorlinks=true,linkcolor=blue,citecolor=red,urlcolor=blue]{hyperref}
\usepackage[left=1.15in,right=1.15in,bottom=1.1in,top=1.1in]{geometry}
\usepackage{tikz,amsthm,amsmath,amstext,amssymb,amscd,epsfig,euscript, mathrsfs, dsfont,pspicture,multicol, mathtools,graphpap,graphics,graphicx,times,enumerate,subfig,sidecap,wrapfig,color}
\usepackage{url}
\makeatletter
\def\@tocline#1#2#3#4#5#6#7{\relax
  \ifnum #1>\c@tocdepth 
  \else
    \par \addpenalty\@secpenalty\addvspace{#2}%
    \begingroup \hyphenpenalty\@M
    \@ifempty{#4}{%
      \@tempdima\csname r@tocindent\number#1\endcsname\relax
    }{%
      \@tempdima#4\relax
    }%
    \parindent\z@ \leftskip#3\relax \advance\leftskip\@tempdima\relax
    \rightskip\@pnumwidth plus4em \parfillskip-\@pnumwidth
    #5\leavevmode\hskip-\@tempdima
      \ifcase #1
       \or\or \hskip 1em \or \hskip 2em \else \hskip 3em \fi%
      #6\nobreak\relax
    \dotfill\hbox to\@pnumwidth{\@tocpagenum{#7}}\par
    \nobreak
    \endgroup
  \fi}
\makeatother

\DeclareGraphicsExtensions{.png}

\newcommand{\bela}{\begin{equation} \label}
\newcommand{\eeq}{\end{equation}}
\newcommand{\ba}{\begin{array}}
\newcommand{\ea}{\end{array}}

\newtheorem{theorem}{Theorem}[section]
\newtheorem{proposition}{Proposition}[section]
\newtheorem{lemma}{Lemma}[section]
\newtheorem{corollary}{Corollary}[section]
\theoremstyle{remark}
\newtheorem{remrak}{Remark}[section]

\catcode`\@=12
 
\begin{document}
\title[Magnetic Robin Laplacian  ]{Magnetic perturbations of the Robin Laplacian in the strong coupling limit}
\author[]{Rayan Fahs}
\address{LAREMA,  Faculté  des  Sciences,  2  Boulevard  Lavoisier,  Université  d’Angers,  49045  Angers,  France.}
\email{rayan.fahs@univ-angers.fr}
\subjclass[2010]{35P15, 47A10, 47F05}
\keywords{Magnetic Robin Laplacian, semi classical analysis, eigenvalues, Born-Oppenheimer approximation.}
\maketitle
\begin{abstract}
This paper is devoted to the asymptotic analysis of the eigenvalues of the Laplace operator with a strong magnetic field and Robin boundary condition on a smooth planar domain and with a negative boundary parameter. We study the singular limit when the Robin parameter tends to infinity which is equivalent to a semi-classical limit involving a small positive semi-classical parameter. The main result is a comparison between the spectrum of the Robin Laplacian with an effective operator defined on the boundary of the domain via the Born-Oppenheimer approximation. More precisely, the low-lying eigenvalue of the Robin Laplacian is approximated by the those of the effective operator. When the curvature has a unique non-degenerate maximum, we estimate the spectral gap and find that the magnetic field does not contribute to the three-term expansion of the eigenvalues. In the case of the disc domains, the eigenvalue asymptotics displays the contribution  of the magnetic field explicitly.
\end{abstract}
\section{Introduction}
\subsection{Magnetic Robin Laplacian  }
Let $ \Omega$ be a bounded open subset of $ \mathbb{R}^{2}.\,$ It is assumed that the boundary $\Gamma=\partial\Omega$  is smooth $\mathit{C}^{\infty}$. In this paper, we study the eigenvalues of the magnetic Robin Laplacian in \textit{$L^{2}( \Omega )$} with a large negative parameter, in a continuation of the works initiated in \cite{7, 18, 2}. The  operator is
\begin{align*}
\textit{$ \mathcal{P}^{b}_{\gamma } $ }&=  -( \nabla -ib\textbf{$A_{0}$})^{2}\,,
\end{align*}
with domain $$\mathrm{Dom}(\textit{$\mathcal{P}^{b}_{\gamma}$}) =  \lbrace u \in  \textit{$H^{2}(\Omega)$} : \nu\cdot(\nabla -ib\textbf {$ A_{0}$})u +\gamma \,u =0 \ \text{\,\,on}\ \Gamma  \rbrace \,, $$ where
\begin{enumerate}
  
    \item[$-$]  $\textbf {$ A_{0}$}\,$ is the vector potential,
    \item[$-$]$ \nu$ is the unit outward normal vector of $\Gamma$,
    \item[$-$]$ \gamma<0\,$ is the Robin parameter,
    \item[$-$] $b>0$ is the intensity of the applied magnetic field,
     \item[$-$]$\,\nu\cdot(\nabla -ib\textbf {$ A_{0}$})u +\gamma \,u =0 \ \text{ on}\  \Gamma\, $ is the Robin boundary condition.
    \end{enumerate} 
The magnetic potential $ \textbf {$ A_{0}$}\,$ generates a constant magnetic field $B$ equal to 1 and it is defined by
$$ \textbf {$ A_{0}$}(x_{1},x_{2}) =( \textbf {$ A_{0_{1}}$}, \textbf {$ A_{0_{2}}$}):= \dfrac{1}{2}(-x_{2},x_{1})\,.$$ 
  We have: \begin{align*}
    B\,&:=\nabla\times\textbf {$ A_{0}$}=\partial_{x_{1}}\textbf {$ A_{0_{2}}$}-\partial_{x_{2}}\textbf {$ A_{0_{1}}$}=1\,.
    \end{align*}
    
To be more precise, the magnetic Robin Laplacian $\, \textit{$\mathcal{P}^{b}_{\gamma}$}\, $ is defined via the Lax-Milgram theorem, from the closed semi-bounded quadratic form  \cite[Ch. 4]{5}
$$ \textit{$H^{1}(\Omega)$} \ni u \mapsto \textit{$\mathcal{Q}^{b}_{\gamma } $ }(u) := \displaystyle{\int_{ \Omega }}  |(\nabla -ib\textbf {$ A_{0}$})u(x)|^{2}\mathrm{~d} x + \gamma  \displaystyle{\int_{ \Gamma}} |u(x)|^{2} \mathrm{~d} s(x)\, , $$
where $\mathrm{~d}s\,$ is the standard surface measure on the boundary.
Note that, by a classical trace theorem (see for instance  \cite{16}  in the case of straight boundary), the trace of $\, u \,$ is well defined as an element of \textit{$H^{\frac{1}{2}}(\Gamma)$}\,, and the quadratic form $\textit{$\mathcal{Q}^{b}_{\gamma } $ }$ is   well defined and bounded from below.\\

Since $ \Omega $ is bounded and regular, then  the embedding of $ \textit{$H^{1}(\Omega)$}$ into $ \textit{$L^{2}(\Omega)$}$ is compact, therefore the operator  $ \textit{ $  \mathcal{P}^{b}_{\gamma } $ }$ has a compact resolvent. Its spectrum is purely discrete, and since  $ \textit{ $  \mathcal{P}^{b}_{\gamma } $ }$ is self-adjoint and bounded  from below, it consists of an increasing sequence of eigenvalues. Our aim is to examine the magnetic effects on the principal eigenvalues $\big( \lambda_{n}(b,\gamma)\big)_{n\in \mathbb{N}^{*}}\,,$ when the Robin parameter $\gamma$ tends to $-\infty\,$ and $b$ tends to $+\infty\,$ simultaneously.\\

If $\Omega$ is simply connected, the eigenvalue $\lambda_{n}(b,\gamma)$ is independent of the choice of the vector potential $\textbf {$ A_{0}$}$ of the magnetic field. This is a consequence of invariance under gauge transformations; if $\text{A} \in  \textit{$H^{1}(\Omega; \mathbb{R}^{2})$}$ and $\text{curl} A=1$, then $\text{A}=\textbf {$ A_{0}$}+\nabla\phi$  for a function $\phi \in \textit{$H^{2}(\Omega)$}$ (cf. \cite[Props D.1.1 and D.2.1]{3}), and in turn
$$e^{ib\phi}( \nabla -ib\textbf{$A_{0}$})^{2}e^{-ib\phi}=( \nabla -ib\textbf{$A_{0}$}-ib\nabla\phi)^{2}=( \nabla -ib\textbf{$A$})^{2}. $$

 A part from its own interest, the study of the spectrum of the operator $\mathcal P^b_\gamma$ arises in several contexts,  where both situations subject to  magnetic fields,  $b>0$, or without magnetic fields, $b=0$, occur.
Estimating  the ground  state energy of $\mathcal P_\gamma^b$, leads to information on the critical temperature/critical fields of certain superconductors surrounded by other materials \cite{2, 21}. On the other hand, eigenvlaue asymptotics in the singular limit $\gamma\to-\infty$, provide counter examples in the context of spectral geometry \cite{22, 12}, and has connections to the study of Steklov eigenvalues \cite{23}.
The case $\gamma=0$ corresponds to the Neumann magnetic Laplacian, wich has been studied in many papers \cite{31,6}.

\subsection{Main results.}\label{s1}
In the case without magnetic field, $ b=0 $, the asymptotic expansion of the eigenvalues of the Robin Laplacian has been the subject of recent studies \cite{9, 7, 11, 12, 14, 18,  19}. In particular, the derivation  of an effective operator on the boundary, involving the Laplace-Beltrami operator as well as the mean curvature of the boundary. In the same spirit, we will study the negative spectrum of the Robin's Laplacian on a bounded domain with Robin condition at the boundary. Our aim is to improve/complement earlier estimates and also clarify the magnetic field's  contribution in the spectral asymptotic.\\

 The main contribution of this article is Theorem~\ref{Q2} below, which involves an effective operator that we introduce as follows. We denote by $|\Omega|$  the  area of the domain $\Omega$  and by $|\partial\Omega|$ the arc-length of its boundary. Let us parameterize the boundary, $\partial\Omega$,  by the arc-length, which we denote  by $s$. Let $\kappa(s)$ be the  curvature of $\partial\Omega\,$ at the point defined by the arc-length $s$, with the convention that $\kappa(s)\geq 0$ when $\partial\Omega$ is  convex in a neighborhood of $s$, and negative otherwise. The effective operators are self-adjoint operators in the Hilbert space $L^2(\partial\Omega)$,  acting on periodic functions, and defined as follows
\begin{equation}\label{eq:def-efop}
\begin{aligned}
&\mathcal{L}_{\gamma,\beta_{0}}^{\mathrm{eff}}=-\gamma^{-2}(\partial_{s}-ib\beta_{0})^{2}-1+ \gamma^{-1}\kappa(s) \\
&{\rm Dom}\big( \mathcal{L}_{\gamma,\beta_{0}}^{\mathrm{eff}}\big)=H^{2}(\, \mathbb{R}/ |\partial\Omega|\mathbb{Z}\,)
\end{aligned}
\end{equation} 
 with $\beta_0$ the circulation of the magnetic field,  
\begin{equation}\label{2é}
    \beta_{0}=\dfrac{|\Omega|}{|\partial\Omega|}\,.
\end{equation}
{By the min-max principle, and standard semi-classical approximation,
\begin{equation}\label{eq:eff-op-asymp}
 \lambda_n(\mathcal{L}_{\gamma,\beta_{0}}^{\mathrm{eff}})=-1+\kappa_{\max}\gamma^{-1}+o(\gamma^{-1}) \,, \end{equation}  }
 where
 \begin{equation}\label{NR}
    \kappa_{\max}=\max_{\partial\Omega}\kappa(s)\,. 
 \end{equation}
%
 %
We will describe the asymptotics of the negative eigenvalues of  $ \textit{ $  \mathcal{P}^{b}_{\gamma } $ }$, via those of the  effective operator on the boundary. More precisely, the $n$-th eigenvalues for the Laplace operator  $\mathcal P_\gamma^b$, is approximated  by the  $n$-th eigenvalue of the operators $\mathcal{L}_{\gamma,\beta_{0}}^{\mathrm{eff}}$. We will use the following notation:
\begin{equation}\label{eq:t-O}
M=\tilde{\mathcal O} (|\gamma|^\beta)\quad{\rm  iff}\quad \forall\,\epsilon>0, M=\mathcal O(|\gamma|^{\beta-\epsilon}).
\end{equation}
\begin{theorem}
\label{Q2}
 Let $0 \leq \eta <\frac32 $ and  $\, 0< c_{1}<c_{2}$\,. Suppose that 
 \[c_{1}|\gamma|^{\eta} \leq b \leq c_{2}|\gamma|^{\eta}\,. \]
 As  $\gamma$ tends to $-\infty\,,$ for any $n \in \mathbb{N}^{*}\,,$ we have
$$ |\lambda_{n}(b,\gamma)-\gamma^{2} \lambda_{n}(\,\mathcal{L}_{\gamma,\beta_{0}}^{\mathrm{eff}})|
=\tilde {\mathcal O}(\gamma^{2(\eta-1)})\,,$$ 
where $\,\lambda_{n}(\mathcal{L}_{\gamma,\beta_{0}}^{\mathrm{eff}})$ is the $n$-eigenvalues of 
the  effective operator introduced in \eqref{eq:def-efop}.
\end{theorem}
{
It results from Theorem~\ref{Q2} and \eqref{eq:eff-op-asymp},
\begin{equation}\label{eq:2terms-asym}
\lambda_{n}(b,\gamma)=-\gamma^{2}+\kappa_{\max}\gamma+\tilde{\mathcal{O}} (\gamma^{2(\eta-1)})\,. 
\end{equation}
When  $b$ is a fixed constant (i.e. $\eta=0$), \eqref{eq:2terms-asym} is known \cite[Prop.~5.1]{30} and in the absence of the magnetic field ($b=0$)  \cite[Thm.~1.1]{7} and \cite[Thm.~1]{32} .
Notice that, the remainder term is of lower order as long as $\eta<\frac32$.
If $\eta=\frac32$, we have by \cite{18},
\[ \lambda_{n}(b,\gamma)=-\gamma^{2}+\bigg(\kappa_{\max}+\dfrac{b^{2} \gamma^{-3}}{4}\bigg)\gamma+o(\gamma)\,.\] }

In  a generic situation where the curvature has a unique non-degenerate maximum, Theorem~\ref{Q2} allows us to determine the leading order term of the spectral gap between the successive eigenvalues. This is valid under the following assumption:
$$\textbf{Assumption A\,\,\,}\left\{\begin{array}{l}\text { $\kappa$ attains its maximum } \kappa_{\text {max }} \text { at a unique point; } \\ \text { the maximum is non-degenerate, i.e. } \kappa^{\prime \prime}(0)<0\, ,\end{array}\right.$$
where $\kappa_{\text{max}}$ is introduced in \eqref{NR}, denotes the maximal curvature along the boundary  $\partial\Omega$ and the arc-length parametrization of the boundary is selected  so that  $\kappa(0)=\kappa_{\text{max}}\,$.\\

Under \textbf{Assumption A\,}, we study the three terms of the   eigenvalue asymptotics  of the effective operator and find that
\begin{equation}\label{eq:eff-asymp-sg} \lambda_{n}(\,\mathcal{L}_{\gamma,\beta_{0}}^{\mathrm{eff}})= -1+\kappa_{\max}\gamma^{-1}+(2n-1)\sqrt{\frac{-\kappa^{\prime \prime}(0)}{2}}|\gamma|^{-3/2}+o(|\gamma|^{-3/2})\,. 
\end{equation}
This yields the following corollary of Theorem~\ref{Q2}:
\begin{corollary}
\label{abc}
Let $\, 0\leq \eta<\frac{5}{4}\,$ and  $0< c_{1}<c_{2}$\,. Suppose that 
 \[c_{1}|\gamma|^{\eta} \leq b \leq c_{2}|\gamma|^{\eta}\,. \]
 Under the \textbf{Assumption A}, as  $\gamma$ tends to $-\infty\,,$ for any $n \in \mathbb{N}^{*}\,,$ we have
$$\lambda_{n}(b,\gamma)=-\gamma^{2}+\kappa_{\text{max}} \gamma+(2n-1)\sqrt{\frac{-\kappa^{\prime \prime}(0)}{2}}|\gamma|^{1/2} +o(|\gamma|^{1/2})\,.$$
\end{corollary}
 { The technical condition on $\eta$ appears as follows. After inserting \eqref{eq:eff-asymp-sg} into the estimate of Theorem~\ref{Q2}, we get the error $\tilde O(\gamma^{2(\eta-1)})$, which  is of  order $o(|\gamma|^{1/2})$ if and  only if $\eta<\frac54$. }

In the asymptotics  of Corollary~\ref{abc}, the dependence with respect to  the labeling of the eigenvalues, $n$, appears in the third term of the expansion. However,  in the generic situations described by \textbf{Assumption A\,}, the contribution of the magnetic field is  hidden in the remainder term, and is of lower order compared to that of the curvature. Corollary~\ref{abc} follows from the spectral asymoptotics of the  1D effective operators, which is valid under a weaker assumption on the strength  of the magnetic field ($b=\mathcal O(|\gamma|^{5/4})$) when compared to the  one in Theorem~\ref{Q2} ($b=\mathcal O(|\gamma|^{3/2})$).\\

In the  case of disc domains, \textbf{Assumption A\,\,\,} fails, but we can improve Theorem~\ref{Q2} and display the influence of the magnetic field in  the lowest eigenvalue asymptotics.
\begin{theorem}
\label{86}
 Assume that $\Omega=D(0,1)$  is the unit disc and that $ \,0\leq \eta<1 $. We  have, as $\gamma\to-\infty$,
$$\lambda_{1}(b,\gamma)=-\gamma^{2}+\gamma+\left(\inf _{m \in \mathbb{Z}}\left(m-\frac{b}{2}\right)^{2}-\frac{1}{2}\right)+o(1)\,,$$
uniformly with respect  to $b$ up to $b=\mathcal O(|\gamma|^\eta)$.
\end{theorem}
Theorem~\ref{86}  was known when $b$ is fixed \cite[Thm. 1.1]{2}.  The proof of Theorem~\ref{86} relies on  the derivation  of  the following effective operator, which is more accurate than the one in \eqref{eq:def-efop} (note that in the unit disc, the curvature is constant and equal to 1),
\begin{equation}\label{eq:def-efop*}
\mathcal{L}_{\gamma,\rm disc}^{\mathrm{eff}} =-\gamma^{-2}(\partial_{s}-ib\beta_{0})^{2}-1+ \gamma^{-1} -\frac{1}{2}\gamma^{-2}\,.
\end{equation}

 \subsection{Organization of the paper.} The paper is organized as follows. In Section \ref{103}, we introduce an effective semiclassical parameter, we discuss a semi-classical version of the operator and we recall why the eigenfunctions are localized, via Agmon estimates near the boundary. As a consequence, we replace the initial problem by a problem on a thin tubular neighborhood of the boundary. In Section  \ref{104}, by using the Born-Oppenheimer approximation, we derive an effective operator whose eigenvalues simultaneously describe the eigenvalues of the magnetic Robin Laplacian. In Section  \ref{105},  we show that the
bound states for the effective operator decay exponentially away from points of maximal curvature. Then we reduce the study to a free-flux operator, which is a perturbation of the harmonic oscillator. Moreover, we estimate the eigenvalues for the effective operator with large magnetic field, thereby proving Corollary \ref{abc}.  Finally, in Section  \ref{5é}, In the case of the disc domains, we describe the term which determines the influence of the magnetic field on the spectrum of the Robin Laplacian, thereby proving Theorem \ref{86}.
 In Appendix \ref{oi1}, we recall the known results related to a family of one dimensional auxiliary differential operators.  
\section{Boundary coordinates}\label{bc}

The key to proving Theorem \ref{Q2} \ is a reduction to the boundary. Indeed, the eigenfunctions associated with eigenvalues lower than $ - \epsilon \, h \, $ with $\epsilon>0$ concentrate exponentially near the boundary (cf. Proposition \ref {agmon}). 
To single out the influence of the boundary curvature, we need a special coordinate system displaying the arc-length along the boundary and the normal distance to the boundary. We will refer to such as boundary coordinates. 
In this section, we introduce the necessary notation  to use these coordinates (cf. \cite{13,3}). Let
$$ \mathbb{R}/(|\partial \Omega|\mathbb{Z})\ni s \longmapsto M(s)\in \partial\Omega\,$$
be the arc-length parametrization of $\partial \Omega$. We will always work with $|\partial\Omega|$-periodic functions sometimes restricted to the interval
 $$ \left]-\frac{|\partial\Omega|}{2} , \frac{|\partial\Omega|}{2}\right]=]-L,L\,]\,.$$
At the point $M(s) \in \partial \Omega, T(s)=M^{\prime}(s)$ is the unit tangent vector and $\nu(s)$ is the unit  normal vector such that
$$
\forall s \in \mathbb{R} /(|\partial \Omega| \mathbb{Z}), \quad \operatorname{det}(T(s), \nu(s))=1 \,.
$$
The curvature $\kappa(s)$ at point $M(s)$ is then defined as follows
$$
T^{\prime}(s)=\kappa(s) \nu(s)\,.
$$
The smoothness of the boundary yields the existence of a constant $t_{0}>0$ such that, upon defining$$
V_{t_{0}}=\left\{x \in \Omega: \operatorname{dist}(x, \partial \Omega)<t_{0}\right\}
$$
 the change of coordinates
$$
\Phi: \mathbb{R} /(|\partial \Omega| \mathbb{Z}) \times\left(0, t_{0}\right) \ni(s, t) \mapsto x=M(s)-t \nu(s) \in V_{t_{0}}
$$
becomes a diffeomorphism. Let us note that, for $x \in V_{t_{0}}$, one can write
$$
\Phi^{-1}(x):=(s(x), t(x)) \in \mathbb{R} /(|\partial \Omega| \mathbb{Z}) \times\left(0, t_{0}\right)
$$
where $t(x)=\operatorname{dist}(x, \partial \Omega)$ and $s(x) $ is the coordinate of the point $M(s(x))\in \mathbb{R} /(|\partial \Omega| \mathbb{Z})$ satisfying the relation $\operatorname{dist}(x, \partial \Omega)=$
$|x-M(s(x))|\,.$
\\
The inverse of $\Phi$ defines a coordinate system on a tubular neighborhood of $\partial\Omega $ in $\overline{\Omega}$ that can be used locally or semi-globally. Now we express various integrals in the new coordinates $(s, t)$.
\\\\
For all $ s, $ the Jacobian matrix $ J_{\Phi} $ is written, as the $2\times2$ matrix with column vectors $M^{\prime}-t \nu^{\prime}$ and $\nu$
$$
J_{\Phi}(s, t)=\left(M^{\prime}(s)-t \nu^{\prime}(s), \nu(s)\right)\,.
$$
As $ \nu^{\prime}(s)=\kappa(s) M^{\prime}(s),$ the determinant of the Jacobian matrix of the transformation $ \Phi^{- 1} $ is given by: $$\mathrm{det} J_{\Phi}(s,t) = 1-t \kappa(s). $$
In the new coordinates, the components of the vector field $\mathbf{A}_{0}$ are given as follows,
$$
\begin{array}{l}
\tilde{A}_{1}(s, t)=(1-t \kappa(s)) \mathbf{A}_{0}(\Phi(s, t)) \cdot M^{\prime}(s)\,, \\
\tilde{A}_{2}(s, t)=\mathbf{A}_{0}(\Phi(s, t)) \cdot \nu(s)\,.
\end{array}
$$
The new magnetic potential $\tilde{A}_{0}=(\tilde{A}_{1},\tilde{A}_{2})$ satisfies,
$$\bigg[\dfrac{\partial \tilde{A}_{2} }{\partial s}(s,t)- \dfrac{\partial \tilde{A}_{1} }{\partial t}(s,t)\bigg] ds \wedge dt =\text{curl} A_{0}(\Phi^{-1}(s, t))dx \wedge dy=(1-t \kappa(s))ds \wedge dt\,.$$
For all  $ u \in \textit{$L^{2}( V_{t_{0}})$}$, we assign the pull-back function $\tilde{u}$   defined in the new coordinates as follows $$\tilde{u} =u\circ \Phi\,. $$
Consequently, for all  $\, u \in \textit{$H^{1}( V_{t_{0}})$}\, ,$ we have
\begin{align*}
  &  \displaystyle{\int_{V_{t_{0}}}}|(\nabla -ib\textbf {$ A_{0}$})u(x)|^{2} \mathrm{~d} x \\&\quad\quad\quad= \displaystyle{\int_{\Phi^{-1}\left(V_{t_{0}}\right)}}\left[ \,(1-t\kappa(s))^{-2}|(\partial_{s}-ib\tilde{\textbf {$ A_{1}$}}){\tilde{u}}|^{2}+|(\partial_{t}-ib\tilde{\textbf {$ A_{2}$}}){\tilde{u}}|^{2} \right] (1-t\kappa(s))\mathrm{~d} t\mathrm{d} s \,,
\end{align*}

 $$ \int_{V_{t_{0}}}|u(x)|^{2} \mathrm{~d} x = \displaystyle{\int_{\Phi^{-1}\left(V_{t_{0}}\right)}} |{\tilde{u}}|^{2}(1-t\kappa(s))\mathrm{~d} t\mathrm{d} s\, ,$$
and
$$ \int_{V_{t_{0}}\cap \,\partial\Omega}|u(x)|^{2} \mathrm{~d} x = \int |\tilde{u}(s,t=0)|^{2}\mathrm{~d} s\, . $$ 
\\
The following lemma can be found in \cite{4} .
\begin{lemma}
\label{tutu}
 For $t_{0}>0\,$ small enough, there exists a gauge transformation  $\varphi(s,t)$ on  $\mathbb{R} /(|\partial \Omega| \mathbb{Z}) \times\left(0, t_{0}\right)$ such that   $ \bar{\textbf {$ A_{0}$}}$ defined by
$$ \bar{\textbf {$ A_{0}$}}=\tilde{\textbf {$ A_{0}$}}-\nabla_{(s,t)}\varphi\,, $$ satisfies
$$\bar{\textbf {$ A_{0}$}}(s,t)=
\begin{pmatrix}
   \bar{\textbf {$ A_{1}$}}(s,t)\\\bar{\textbf {$ A_{2}$}}(s,t)
\end{pmatrix}=
\begin{pmatrix}
  \beta_{0} -t+\dfrac{t^{2}}{2} \kappa(s) \\0
\end{pmatrix}\,,$$
with $$\beta_{0} =\dfrac{1}{|\partial\Omega|}\displaystyle\int_{\Omega}\mathrm{curl}\textbf {$ A_{0}$} \,dx =\dfrac{|\Omega|}{|\partial\Omega|}\,, $$
 and for all  $\, u \in \textit{$H^{1}( V_{t_{0}})$}\, ,$ we have $$
\displaystyle{\int_{V_{t_{0}}}}|(\nabla -ib\textbf {$ A_{0}$})u(x)|^{2} \mathrm{~d} x =
\int_{\Phi^{-1}\left(V_{t_{0}}\right)}\left((1-t k(s))^{-2}\left|\left( \partial_{s}-ib\bar{\textbf {$ A_{1}$}}\right) w\right|^{2}+\left|\partial_{t} w\right|^{2}\right)(1-t k(s)) \mathrm{~d}  s \mathrm{d} t
$$
where  $w=\mathrm{e}^{i\varphi} \tilde{u}$  and $\tilde{u}=u \circ \Phi \text {.}$
\end{lemma}
\section{Preliminaries}
\subsection{Transformation into a semi-classical problem}
\label{103}
We will prove the results of Section \ref {s1} in a semiclassical framework. Let us consider the semi-classical parameter
$$h = \gamma^{-2}\,.$$
 The limit $\gamma \rightarrow -\infty$ is now equivalent to the semi-classical limit $h \rightarrow 0^{+}$. The quadratic form can be written as
 $$\textit{$\mathcal{Q}^{b}_{\gamma } $}(u)=h^{-2}\left(\, \displaystyle{\int_{ \Omega }}|(h\nabla -ibh\textbf {$ A_{0}$})u(x)|^{2} \mathrm{~d} x - h^{\frac{3}{2}}\displaystyle{\int_{ \Gamma} }|u(x)|^{2} \mathrm{~d} s(x)\,\right):= h^{-2} \textit{$\mathcal{Q}^{b}_{h} $}(u)  \,.$$
Consequently, we obtain the self-adjoint operator depending on $h$
$$ \textit{$ \mathcal{L}^{b}_{h} $ } =  -( h \nabla -ih b\textbf{$A_{0}$})^{2}\,,$$  
with domain
$$\mathrm{Dom}(\textit{$\mathcal{L}^{b}_{h}$})= \lbrace\, u \in  \textit{$H^{2}(\Omega)$} : \nu.(\nabla -ib\textbf {$ A_{0}$})u -h^{\frac{1}{2}} u =0\,\, \ \text{on}\  \Gamma  \, \rbrace \,.$$
Clearly, $$\textit{ $ \mathcal{P}^{b}_{\gamma} $} =h^{-2} \textit{$\mathcal{L}^{b}_{h}$} \,,$$ and thus the  relation between the spectra of the operators
 $ \textit{$\mathcal{P}^{b}_{\gamma}$}$  and  $\textit{$ \mathcal{L}^{b}_{h}$}$ is displayed as follows : 
 \begin{equation}
 \label{90}
     \sigma(\textit{$ \mathcal{P}^{b}_{\gamma}$}) \, = \, h^{-2} \,\,  \sigma(\textit{$ \mathcal{L}^{b}_{h} $})\,.
 \end{equation}
Let $(\mu_{n}(h,b))_{n\in \mathbb{N}^{*}}\,$ be the sequence of eigenvalues of the operator  \textit{$ \mathcal{L}^{b}_{h} $ }.  Theorem \ref{Q2} and Corollary \ref{abc} can be reformulated in semi-classical form as follows.
 \begin{theorem}
 \label{R1}
  Let $n\in  \mathbb{N}^{*}\,,$ $ \eta  < \frac32 \,$ and  $\,0< c_{1}<c_{2}$\,. There exist $h_{0}\,>0,\,$  such that, if $h\in(0,h_{0})\,$ and $b$ satisfies  $$c_{1}h^{\frac{-\eta}{2}}\leq b \leq c_{2}h^{\frac{-\eta}{2}}\,, $$ then
we have : 
 $$| \mu_{n}(h,b)-   h\,\lambda_{n}(\mathcal{L}_{h,\beta_{0}}^{\mathrm{eff}})    |= \tilde{\mathcal{O}}(h^{3-\eta})\,,   $$
 with
$$\mathcal{L}_{h,\beta_{0}}^{\mathrm{eff}}=-h(\partial_{s}-ib\beta_{0})^{2}-1-\kappa(s) h^{\frac{1}{2}}\,.$$
 \end{theorem} 
\begin{remrak}
We distinguish two important cases of $\eta \,:$
     \item  \quad \quad \quad - If $\,0\leq\eta<\frac{5}{4}$, then $,\,\tilde{\mathcal{O}}(h^{3-\eta})=o(h^{\frac{7}{4}})\,.$
     \item  \quad \quad \quad - If $\,\frac{5}{4} \leq\eta<\frac{3}{2}$, then $,\,\tilde{\mathcal{O}}(h^{3-\eta})=o(h^{\frac{3}{2}})\,.$
 \end{remrak}
 By the method of harmonic approximation, we deduce:
 \begin{corollary}
 \label{ab}
 Let $n \in \mathbb{N}^{*}\,, $ $0 \leq \eta < \frac{5}{4} \,$ and  $\,0< c_{1}<c_{2}$\,. Under the \textbf{Assumption A}, there exist  $h_{0}\,>0,\,$ for all $h\in(0,h_{0})\, $ and $b$ satisfies  $$c_{1}h^{\frac{-\eta}{2}}\leq b \leq c_{2}h^{\frac{-\eta}{2}}\,, $$ we have : 
  $$\mu_{n}(h,b)=-h-\kappa_{\text{max}} h^{3/2}+(2n-1)\sqrt{\frac{-\kappa^{\prime \prime}(0)}{2}} h^{7/4}+ o(h^{\frac{7}{4}})\,.$$
 \end{corollary}
 In the case of the disc, we derive the following effective operator
\begin{equation}\label{eq:def-efop1*}
\mathcal{L}_{h,\rm disc}^{\mathrm{eff}} =-h(\partial_{s}-ib\beta_{0})^{2}-1- h^{\frac{1}{2}}-\frac{h}{2}\,.
\end{equation}
As a consequence, we get the following theorem.
 \begin{theorem}
\label{86a}
 Assume that $\Omega=D(0,1)$  is the unit disc$, \,0\leq \eta<1 $ and  that $0< c_{1}<c_{2}$\,. There exist  $h_{0}\,>0,\,$ for all $h\in(0,h_{0})\, $  and $b$ satisfies 
 $$c_{1}h^{\frac{-\eta}{2}}\leq b \leq c_{2}h^{\frac{-\eta}{2}}\,, $$  we have : 
$$\mu_{n}(h,b)=-h-h^{\frac{3}{2}}+\left(\inf _{m \in \mathbb{Z}}\left(m-\frac{b}{2}\right)^{2}-\frac{1}{2}\right) h^{2}+o(h^{2})\,.$$
\end{theorem}
\subsection{Reduction near the boundary via Agmon estimates.}
The goal of this section is to show that the eigenfunctions associated to eigenvalues less than $ -\epsilon \,h, $ with $0<\epsilon<1$ a fixed constant, are concentrated exponentially near the boundary at the scale $ h^{1/2} \, $ . 
\begin{proposition}
\label{agmon}
 Let $M \in(-1,0)$. For all  $\alpha<\sqrt{-M}$, there exist constants $C>0$ and $h_{0} \in\,(0,1)$ such that, for $h\in (0,h_{0}),\,$ if $u_{h,b}$ is a normalized ground state of $\textit{ $ \mathcal{L}^{b}_{h} $}\,$  with eigenvalue $\mu(h,b)\,$ such that $\,\mu(h,b)< M\,h\,$, then
\begin{equation}
\label{18}
\displaystyle{\int_{ \Omega }}\left(                      |u_{h,b}(x)|^{2}+h|(\nabla -ib\textbf {$ A_{0}$})u_{h,b}(x)|^{2}\right)\exp\left(\dfrac{2\alpha  \mathrm{dist}(x,\Gamma)}{h^{1/2}}\right)\mathrm{~d} x \leq C . 
\end{equation}
\end{proposition}
\begin{proof}
The proof is similar to the one of \cite[Thm. 5.1]{7}, and only notice that it is a consequence modulo a partition of unity with balls of size $Rh^{1/2}$ with $R$ large enough and the fact, if the support of $u$ avoids the boundary, we have $\textit{$\mathcal{Q}^{b}_{h} $}(u) \geq -Ch\,. $
\end{proof}

We record the following simple corollary of Proposition \ref{agmon}.
\begin{corollary}
\label{C1}
Let $M \in\,(-1,0)$ , $\rho\,\in\,(0,\frac{1}{2})$ and $c>0\,$. For all $0<\alpha < \sqrt{-M}$, there exists $C>0$ and  $h_{0} \in\,(0,1)$ such that, if $u_{h,b}$ is a normalized ground state of $\textit{$\mathcal{L}^{b}_{h}$}\,$ with eigenvalue $\mu(h,b)\,$ such that  $\,\mu(h,b)< M\,h\,$, then for all  $h\in(0,h_{0}),$ 
\begin{equation}
\label{23}
\displaystyle{\int_{\mathrm{dist}(x,\Gamma)\,\geq c\, h^{\frac{1}{2}-\rho}}} \big(\,|u_{h,b}(x)|^{2}+h|(\nabla -ib\textbf {$ A_{0}$})\,u_{h,b}(x)|^{2}\,\big)\, dxC \,\leq \,\exp(- 2\,\alpha\,c\,h^{-\rho} )\,.
\end{equation}
\end{corollary}
As a consequence, for small $ \, h $, the ground states  of the operator   $\,\textit{$\mathcal{L}^{b}_{h} $}\,$  are concentrated near the boundary  of $\Omega $ (cf. Corollary
\ref{C1}). This will allow us to work in a tubular neighborhood of $ \partial\Omega \,. $\\

Let $\,\rho\in(0,1/2)\,,$ we introduce the $\delta$-neighborhood of the boundary
$$\,\Omega_{\delta}=\lbrace \,x \in \Omega \, :\,\mathrm{dist}(x,\Gamma)<\delta\,\rbrace\,,\quad \delta =h^{\frac{1}{2}-\rho}\,.$$
 The quadratic form, defined on the variational space
$$ W_{\delta} =\lbrace u \in \textit{$H^{1}(\Omega_{\delta})$} \, :\, u(x)=0\, ,\, \forall \,x \in \Omega \,\text{ such that }  \,\mathrm{dist}(x,\Gamma)=\delta  \rbrace\,,$$
is given  by the formula
\begin{equation}
\label{25}
\textit{$ q^{b,\rho}_{h} $}(u)=\,h^{2}\,\, \displaystyle{\int_{ \Omega_{\delta} }}|(\nabla -ib\textbf {$ A_{0}$})u(x)|^{2} \mathrm{~d} x - h^{\frac{3}{2}}\displaystyle{\int_{\partial \Omega_{\delta}}}|u(x)|^{2} \mathrm{~d} s(x)\,.
\end{equation}
Note again that the trace on  $\partial \Omega_{\delta}$ is well-defined by a classical trace theorem. The associated self-adjoint operator is:
 $$ \textit{$ \mathcal{L}^{b,\rho}_{h} $ }\,= \,-( h \nabla -ih b\textbf{$A_{0}$})^{2}\,,$$  
with domain \begin{align*}
\mathrm{Dom}(\textit{$ \mathcal{L}^{b,\rho}_{h}$})=\, \lbrace u &  \in  \textit{$H^{2}(\Omega_{\delta})$}\,  : \,\nu.(\nabla -ib\textbf {$ A_{0}$})u -h^{\frac{1}{2}} u =0 \,  \text{ on}\ \, \Gamma \,\text{ and}\\& \quad\quad\,\,\,\,\,\,\,\,\, u(x)=0\, ,\, \forall \,x \in \Omega \,\, \,\text{ such that }  \,\mathrm{dist}(x,\Gamma)=\delta   \rbrace \, .
\end{align*}
That is, we consider the realization with mixed boundary conditions (Robin and Dirichlet conditions). Let $(\mu_{n}(h,b,\rho))_{n \in \mathbb{N}^{*}}\,$ be the sequence of eigenvalues of the operator  $\textit{$ \mathcal{L}^{b,\rho}_{h} $ }$. The following proposition  reduces  the analysis to the operator $ \textit{$ \mathcal{L}^{b,\rho}_{h} $ }\,.$
 \begin{proposition}
 \label{L2} Let $ \epsilon >0$ and $\alpha \in(0,\sqrt{\epsilon})\,,$ there exist constants $C>0 \,,\,h_{0}\,\in\,(0,1)$ such that, for all $\, h \,\in (0,h_{0})\,,$ $ n \geq 1$ and $\mu_{n}(h,b)<- \epsilon h$,
 $$\mu_{n}(h,b,\rho) \leq \mu_{n}(h,b)+C\exp(-\alpha h^{-\rho}\,)\,.$$
Moreover, we have, for all $ n \geq 1\,$ and  $h>0$
  $$\mu_{n}(h,b) \leq \mu_{n}(h,b,\rho)\,.$$
\end{proposition}
\begin{proof}  
 
The inequality $\, \mu_{n}(h,b)\, \leq \,\mu_{n}(h,b,\rho)\,$ is not asymptotic.
 Let $(v_{k,h})_{1\leq k \leq n}$ be an orthonormal family of eigenfunctions associated with eigenvalues $(\tilde{\mu}_{k}(h,b,\rho))_{1\leq k \leq n}\,.$ 
We define the function
  $$u_{k,h}    = \begin{cases} v_{k,h}& \text{if $x \in \Omega_{\delta}$ } \\ 0 & \text{if $x \in \Omega \setminus \Omega_{\delta} $} \,. \end{cases}$$ \\
 Let $E$ be the vector subspace generated by the family $(u_{k,h})_{1\leq k \leq n}$ and $u \in E$ which is written as follows:
  $$u =\sum_{k=1}^{n} \alpha_{k} u_{k,h} \,.$$
  Inserting $u$ into the quadratic  form $\textit{$ \mathcal{Q}^{b}_{h}$}(u)$, we obtain
  
  \begin{align*}
      \textit{$ \mathcal{Q}^{b}_{h}$}(u) &=\Big\langle \textit{$ \mathcal{L}^{b,\rho}_{h} $ } \sum_{k=1}^{n} \alpha_{k} v_{k,h},\sum_{k=1}^{n} \alpha_{k} v_{k,h}      \Big \rangle  \\&=\sum_{k=1}^{n} \alpha_{k} \textit{$ q^{b,\rho}_{h} $}(v_{k,h})
      \leq \mu_{n}(h,b,\rho) \sum_{k=1}^{n} \alpha_{k}\|u_{k,h}\|^{2}= \mu_{n}(h,b,\rho)\|u\|^{2}.
  \end{align*}
By the min-max theorem, we get the non-asymptotic inequality $\, \mu_{n}(h,b)\, \leq \,\mu_{n}(h,b,\rho)\,.$\\

Now we determine an upper bound of $\mu_{n}(h,b,\rho)$ in terms of $\mu_{n}(h,b)\,.$
Let $(u_{k,h})_{1\leq k \leq n}$ be an orthonormal family of eigenvectors associated with the eigenvalues $(\mu_{k}(h,b))_{1\leq k \leq n}\,.$
Let us consider the cut off function define on $\mathbb{R}$ as following:
 $$0\leq \chi\leq1 \,\,,\,\, \chi=1 \,\,\text{ on } \,\,]-\infty, 1/2\,]\,\, \text{ and } \,\,\chi=0\,\, \text{ on } \,\,[1,+\infty[\,.$$
 We define the function  $$\chi_{1}(x)=\chi\bigg(  \dfrac{t(x)}{\delta}\bigg)\,.$$
 For $ k=1,...,n\,,\,\,$ we set
 $$v_{k,h}=\chi_{1}u_{k,h}\,.$$ 
 
 Let $F_{h}$ the vector subspace of $\mathrm{Dom}(\textit{$ \mathcal{L}^{b,\rho}_{h}$})$ spanned by the family $(v_{k,h})_{1\leq k \leq n}$
  and $w_{h} \in F_{h} $ which is written as follows:
  $$w_{h} =\sum_{k=1}^{n} \beta_{k} v_{k,h} \,.$$
  The functions $(v_{k,h})$ are almost orthonormal ( Proposition \ref{agmon}).
We note that dim$F_{h}=n$ and $$ \|w_{h}\|^{2}=\sum\limits_{k=1}^{n} |\beta_{k}|^{2}\,+\mathcal{O}(\exp(-\alpha h^{-\rho}\,))\,.$$
Inserting $w_{h}$ into the quadratic  form $\textit{$q^{b,\rho}_{h} $}(w_{h})$ 
 \begin{align*}
 \textit{$q^{b,\rho}_{h} $}(w_{h})=\langle \textit{$ \mathcal{L}^{b,\rho}_{h} $ }w_{h},w_{h} \rangle=\sum_{j,k=1}^{n}\beta_{j}\beta_{k} \langle \textit{$ \mathcal{L}^{b,\rho}_{h} $ }v_{j,h},v_{k,h} \rangle\,.
 \end{align*}
 For $j,k$ fixed and for $ h $ small enough, we have
 \begin{align*}
 \langle \textit{$ \mathcal{L}^{b,\rho}_{h} $ }v_{j,h},v_{k,h} \rangle &= \delta_{j,k}\,\mu_{j}(h,b)-h^{2}  \langle u_{j,h} \Delta \chi_{1}, \chi_{1} u_{k,h}  \rangle -2h^{2}\langle \nabla \chi_{1}.(\nabla -ib \textbf{$A_{0}$})u_{j,h} , \chi_{1} u_{k,h} \rangle\\&\,\,\,\,\,\,+\mathcal{O}(\exp(-\alpha h^{-\rho}\,))\,.
 \end{align*}
From Hölder's inequality and the Corollary \ref{C1}, we obtain
 $$h^{2}  \langle u_{j,h} \Delta \chi_{1}, \chi_{1} u_{k,h}  \rangle +2h^{2}\langle \nabla \chi_{1}.(\nabla -ib \textbf{$A_{0}$})u_{j,h} , \chi_{1} u_{k,h} \rangle = \mathcal{O}(\exp(-\alpha h^{-\rho}\,))\,.$$
Therefore,
 \begin{align*}
 \textit{$q^{b,\rho}_{h} $}(w_{h})&=\sum_{j=1}^{n}\mu_{j}(h,b)|\beta_{j}|^{2}+ \mathcal{O}(\exp(-\alpha h^{-\rho}\,)) \sum_{j,k=1}^{n}\beta_{j}\beta_{k} \\&\leq \bigg(\mu_{n}(h,b) + C\exp(-\alpha h^{-\rho}\,)\bigg) \|w_{h}\|^{2}\,.
 \end{align*}
By the min-max theorem, we have
 $\,\,\,\,\mu_{n}(h,b,\rho) \leq \mu_{n}(h,b)+C\exp(-\alpha h^{-\rho}\,)\,.$
\end{proof}
Proposition \ref{L2} leads us to replace the initial problem by a new Robin-Dirichlet in a $\delta$-neighborhood of the boundary $\Gamma\,.$
\subsection{The Robin Laplacian in boundary coordinates.}\label{tub}
By Lemma \ref{tutu}, the new quadratic form is expressed in tubular coordinates and is written as follows:
  \begin{align*}
 \textit{$\widetilde{q}^{\,b,\rho}_{h}$}(u)&=h^{2}\displaystyle{\int^{L}_{-L}\int^{\delta}_{0}}\,\bigg|\,\bigg(\partial_{s}-ib\beta_{0}-ib\Big(-t+\frac{t^{2}}{2}\kappa(s)\Big)\bigg)\tilde{u}\bigg|^{2}\,(1-t \kappa(s))^{-1}\,\mathrm{~d} t\mathrm{d} s\\& \,\,\,\,\,\,\,\,+h^{2}\displaystyle{\int^{L}_{-L}\int^{\delta}_{0}}|\partial_{t}\tilde{u}|^{2}\,(1-t \kappa(s))\,\mathrm{~d} t\mathrm{~d} s-h^{3/2}\displaystyle{\int^{L}_{-L}}|\tilde{u}(s,t=0)|^{2}\mathrm{~d} s\,,
 \end{align*}
 with $\tilde{u}= e^{i \varphi}u\circ \Phi$   and $L=\dfrac{|\partial\Omega|}{2}\,.$ 
The operator $\textit{ $ \mathcal{L}^{b,\rho}_{h} $}\,$is unitarily equivalent to $ \textit{ $ \widetilde{\mathcal{L}}^{b,\rho}_{h} $}$ the self-adjoint realization on $\textit{$L^{2}((-L,L]\times (0,\delta),a\mathrm{~d} t\mathrm{d} s)$}$, of the differential operator, with $2L-$periodic boundary condition with respect to $s,$
\begin{align*}
 \textit{ $ \widetilde{\mathcal{L}}^{b,\rho}_{h} $}&=-h^{2}a^{-1}\bigg(\partial_{s}-ib\beta_{0}-ib\Big(-t+\frac{t^{2}}{2}\kappa(s)\Big)\bigg)a^{-1}\bigg(\partial_{s}-ib\beta_{0}-ib\Big(-t+\frac{t^{2}}{2}\kappa(s)\Big)\bigg)\\&\,\,\,\,\, -h^{2}a^{-1}\partial_{t}a\partial_{t}\,\,\,\,\,\,\,\,\,\,\,\,  ( \text{in } (\textit{$L^{2}(a\,\mathrm{~d} t\mathrm{d} s))$}\,.
 \end{align*}
 With  $a=1-t \kappa(s)\,.$
In boundary coordinates, the Robin condition  becomes $$ \partial_{t}u(s,t=0)=-h^{-1/2}u(s,t=0)\,.$$
We introduce,
\begin{align*}
&
\,\widetilde{\Omega}_{\delta}=\lbrace (s,t)  :s\in [-L,L[\,\, \text{ and }  \,\,0< t < \delta\,\rbrace\,,\\& \mathrm{Dom}(\textit{$\widetilde{q}^{\,b,\rho}_{h}$})=\lbrace u \in \textit{$H^{1}(\widetilde{\Omega}_{\delta})$} \, :\, u(s,\delta)=0\rbrace\,,\\ &\mathrm{Dom}(\textit{$\widetilde{\mathcal{L}}^{b,\rho}_{h}$})=\, \lbrace u  \in  \textit{$H^{2}(\widetilde{\Omega}_{\delta})$} \,  :u(s,\delta)=0\, \text{ and }\, \partial_{t}u(s,t=0)=-h^{-1/2}u(s,t=0)  \rbrace \, .  
\end{align*}
We know that the eigenfunctions are localized near the boundary, at the scale $ h^{\frac {1}{2}}$ (cf. Prop \ref{agmon}) and in order to obtain a Robin condition independent of $h\,, $ we get a partially semiclassical problem by changing the variable 
$$(s,t) = (s,h^{\frac{1}{2}} \tau).$$
This change of variable transforms the above expression of $ \textit{ $ \tilde{\mathcal{L}}^{b,\rho}_{h} $}$ into an operator  $\textit{ $ \widetilde{\mathcal{L}}^{b,\rho}_{h} $}$ as follows

  \begin{align*}
 &h\bigg[-h\widehat{a}^{-1}\bigg(\partial_{s}-ib\beta_{0}-ib\Big(-h^{\frac{1}{2}} \tau+h\dfrac{\tau^{2}}{2}\kappa(s)\Big)\bigg)\widehat{a}^{-1}\bigg(\partial_{s}-ib\beta_{0}-ib\Big(-h^{\frac{1}{2}} \tau+h\dfrac{\tau^{2}}{2}\kappa(s)\Big)\bigg)\\&\,\,\,\,\,\,\,\,\,\, -\,\widehat{a}^{-1}\partial_{\tau}\widehat{a}\partial_{\tau}\bigg]\,,
 \end{align*}
 where the new weight
  $$\widehat{a}(s,\tau)=1-h^{\frac{1}{2}}\,\tau \kappa(s)\,.$$\\
  The boundary condition becomes
   $$\partial_{\tau}u(s,\tau=0)=-u(s,\tau=0)\,.$$ 
   We consider rather the operator
    $$\textit{$\widehat{\mathcal{L}}^{\,b,\rho}_{h} $}=h^{-1} \textit{$\widetilde{\mathcal{L}}^{\,b,\rho}_{h} $}\,,$$
     and the domain of integration becomes
     $$[-L,L[\times(0,\delta / h^{1/2})=[-L,L[\times(0,h^{-\rho})\,.$$
     The associated quadratic form is
     \begin{align*}
         \textit{$\widehat{q}^{\,b,\rho}_{h}$}(\psi)&=\displaystyle{\int^{L}_{-L}\int^{h^{-\rho}}_{0}}\,\bigg|\,\bigg(h^{1/2}\partial_{s}-ibh^{1/2}\beta_{0}-ibh\Big(-\tau+h^{1/2}\frac{\tau^{2}}{2}\kappa(s)\Big)\bigg)\psi\bigg|^{2}\,\hat{a}^{-1}\,\mathrm{~d} \tau \mathrm{~d} s\\&\quad+\displaystyle{\int^{L}_{-L}\int^{h^{-\rho}}_{0}}|\partial_{\tau}\psi|^{2}\,\hat{a}\,\mathrm{~d} \tau \mathrm{~d} s-\displaystyle{\int^{L}_{-L}}|\psi(s,\tau=0)|^{2}\mathrm{~d} s\,.
     \end{align*}
     We let
     \begin{equation}\label{ay}
\begin{aligned}
&
\,\widehat{\Omega}_{\delta}=\lbrace (s,\tau)  :s\in [-L,L[\,\, \text{ and }  \,\,0< \tau < h^{-\rho}\,\rbrace\,,
\\&\mathrm{Dom}(\textit{$\widehat{q}^{\,b,\rho}_{h}$})=\lbrace u \in \textit{$H^{1}(\widehat{\Omega}_{\delta})$} \, :\, u(s,h^{-\rho})=0\rbrace\,,\\ &\mathrm{Dom}(\textit{$\widehat{\mathcal{L}}^{b,\rho}_{h}$})=\, \lbrace u  \in  \textit{$H^{2}(\widehat{\Omega}_{\delta})$} \,  : u(s,h^{-\rho})=0 \, \text{ and }\, \partial_{\tau}u(s,\tau=0)=-u(s,\tau=0)  \rbrace \, .  
\end{aligned}
   \end{equation}
  Let  $\,\widehat{\mu}_{n}(h,b,\rho)$ be the $n$-th eigenvalue of the self-adjoint operator $\textit{$\widehat{\mathcal{L}}^{\,b,\rho}_{h} $}\,.$ 
 We have
 \begin{equation}
     \label{oloi}
     \widehat{\mu}_{n}(h,b,\rho)=h^{-1} \mu_{n}(h,b,\rho)\,.
 \end{equation}
  \section{Asymptotics of the eigenvalues of the Robin's Laplacian}
  \label{104}
The aim of this section is to prove the following result, which implies Theorem \ref{Q2}.
\begin{theorem}
 \label{R21}
  Let $n \in \mathbb{N}^{*},\,$ $ 0 \leq \eta <\frac{3}{2}  \,,$ $\alpha>0\,$ and $0< c_{1}<c_{2}$\,. There exist $c_{\pm}>0\,$ and $h_{0}\,>0,\,$ such that, if $h\in(0,h_{0})\,$
  and $b$ satisfies
 $$c_{1}h^{\frac{-\eta}{2}}\leq b \leq c_{2}h^{\frac{-\eta}{2}}\,, $$   then we have : 
$$\,\,\,\,\,\, \,\lambda_{n}(\mathcal{L}_{h,\beta_{0}}^{\mathrm{eff},-})+\mathcal{O}(h^{2-\alpha-\eta}) \leq\widehat{\mu}_{n}(h,b,\rho)\leq \,\lambda_{n}(\mathcal{L}_{h,\beta_{0}}^{\mathrm{eff},+})+\mathcal{O}(h^{2-\alpha-\eta})\,,$$
 with
$$\mathcal{L}_{h,\beta_{0}}^{\mathrm{eff},\pm}=-h(1\pm c_{\pm}\,h^{r})(\partial_{s}-ib\beta_{0})^{2}-1-\kappa(s) h^{\frac{1}{2}}-\frac{\kappa(s)^{2}}{2}h\,,$$
and $$r=\min(\alpha,\frac{1}{2})\,.$$
 \end{theorem} 
\subsection{The Feshbach projection}\label{s13}
Following the writing of the magnetic Robin Laplacian in tubular coordinates in Section \ref{tub}, there appears a one-dimensional operator defined, for $s$ fixed, by\\
$$   \begin{cases}
-\,\tilde{a}^{-1}\partial_{\tau}\tilde{a}\partial_{\tau}=-\partial_{\tau}^{2}+h^{1/2} \kappa(s)\tilde{a}^{-1}\partial_{\tau}\\
\partial_{\tau}u(s,\tau=0)=-u(s,\tau=0)\, \text{ in }\textit{$L^{2}((0,h^{-\rho}),(1-\tau h^{1/2}\kappa(s))\,\mathrm{~d} \tau)$}\,.
\end{cases}$$\\
This operator is denoted by
 $$\mathcal{H}_{\kappa(s),h}=\mathcal{H}_{B}^{\lbrace T \rbrace} \,,$$
with $$T=h^{-\rho}\,\,\text{ and }\,\,B=h^{1/2}\kappa(s) \,, $$
and where $\mathcal{H}_{B}^{\lbrace T \rbrace}$ is defined in \ref{d1}.
Let $ v_{\kappa(s),h} $ be a normalized groundstate of  $\mathcal{H}_{\kappa(s),h}$ associated with $\lambda_{1}(\mathcal{H}_{\kappa(s),h})\,.$
 By \cite[Lemma 2.5]{2}, we have 
\begin{equation}\label{4.4}
\lambda_{1}(\mathcal{H}_{\kappa(s),h})=-1-\kappa(s) h^{\frac{1}{2}}-\frac{\kappa(s)^{2}}{2}h+o(h)\,,
\end{equation}
and for $C>0$
 \begin{equation}
     \label{&é}
     \,\,\lambda_{2}(\mathcal{H}_{\kappa(s),h})\geq - Ch^{1/2-\rho}\,.
 \end{equation}
According to Lemma \ref{L)} and Proposition \ref{P1}, there exist constants $C$ and $C_{k}\,,$ such that
\begin{equation}
    \label{RO}
    \|\partial_{s}v_{\kappa(s),h} \|_{\textit{$L^{2}((0,h^{-\rho})$},(1-h^{1/2}\kappa(s)\tau) \mathrm{~d} \tau)}\leq C h^{1/2}\,,
\end{equation}
and
\begin{equation}
    \label{L5}
    \displaystyle{\int_{0}^{h^{-\rho}}} \tau^{k}|v_{\kappa(s),h}|^{2}(1-h^{1/2}\kappa(s)\tau)\mathrm{~d} \tau\leq C_{k}\,.
\end{equation}
We introduce for $s\in[-L,L]$, the Feshbach projection $\Pi_{s}$ on the normalized ground state $v_{\kappa(s),h}$ of $\mathcal{H}_{\kappa(s),h}$
   \begin{equation}
       \label{ZE}
       \Pi_{s}\psi=\langle\psi,v_{\kappa(s),h}\rangle_{\textit{$L^{2}((0,h^{-\rho}),\hat{a}\,\mathrm{~d} \tau)$}} v_{\kappa(s),h}\,.
   \end{equation}
We also let
$$\Pi^{\perp}_{s}=\text{Id}-\Pi_{s}\,,$$
and
$$R_{h}(s)=\|\partial_{s}v_{\kappa(s),h}\|^{2}_{\textit{$L^{2}((0,h^{-\rho}),\hat{a}\,\mathrm{~d} \tau)$}}\,.$$
The quantity $R_{h}$ is sometimes called “Born-Oppenheimer correction”.
\begin{lemma}
\label{L1}
There exist constants $C>0$ and $h_{0}>0$ such that, for all $\psi \in\mathrm{Dom}({\textit{$\widehat{q}^{\,b,\rho}_{h}$}})$ and $h\in(0,h_{0})\,,$ we have:
\begin{align*}
\|[\Pi_{s},\partial_{s}]\psi\|_{\textit{$L^{2}(\hat{a}\mathrm{~d} \tau\mathrm{~d} s)$}}=\|[\Pi^{\perp}_{s},\partial_{s}]\psi\|_{\textit{$L^{2}(\hat{a}\mathrm{~d}\tau\mathrm{~d}s)$}}&\leq \displaystyle\int_{-L}^{L}\bigg[2 R_{h}(s)^{\frac{1}{2}}+ch^{1/2} \bigg]\|\psi\|_{\textit{$L^{2}(\hat{a}\mathrm{~d} \tau)$}}\mathrm{~d} s\,.
\end{align*}
\end{lemma}
\begin{proof}
We estimate the commutator:
\begin{align*}
[\Pi_{s},\partial_{s}]\psi&=-[\Pi^{\perp}_{s},\partial_{s}]\psi\\&=\langle \partial_{s}\psi,v_{\kappa(s),h}\,\rangle_{\textit{$L^{2}(\hat{a}\mathrm{~d}\tau )$}}\, v_{\kappa(s),h}-\partial_{s}\bigg(\langle \psi,v_{\kappa(s),h}\,\rangle_{\textit{$L^{2}(\hat{a}\mathrm{~d} \tau)$}}\, v_{\kappa(s),h}\bigg)
\\&=-\langle \psi,\partial_{s}v_{\kappa(s),h} \rangle_{\textit{$L^{2}(\hat{a}\mathrm{~d} \tau)$}}\, v_{\kappa(s),h}-\langle \psi,v_{\kappa(s),h} \rangle_{\textit{$L^{2}(\tilde{a}\mathrm{~d} \tau)$}}\,\partial_{s} v_{\kappa(s),h}\\&\,\,\,\,\,\,\,\,\,\,\,\,\,\,\,\,\,\,\,\,\,\,\,\,\,\,\,\,\,\,\,\,\,\,\,\,\,\,\,\,\,\,\,\,\,\,\,\,\,\,\,\,\,\,\,\,\,\,\,\,\,\,\,\,\,\,\,\,\,\,\,\,\,\,\,\,\,\,\,\,+h^{1/2}\kappa^{'}(s)\bigg( \displaystyle\int_{0}^{h^{-\rho}}\psi v_{\kappa(s),h} \tau \mathrm{~d} \tau\bigg)v_{\kappa(s),h}\,.
\end{align*}
We get, thanks to the Cauchy-Schwarz inequality and the inequality \eqref{L5}
\begin{align*}
\|[\Pi_{s},\partial_{s}]\psi\|_{\textit{$L^{2}(\hat{a}\mathrm{~d} \tau\mathrm{~d} s)$}}&\leq \displaystyle\int_{-L}^{L}\bigg[2 R_{h}(s)^{\frac{1}{2}}+ch^{1/2} \bigg]\|\psi\|_{\textit{$L^{2}(\hat{a}\mathrm{~d} \tau)$}}\mathrm{~d} s\,.
\end{align*}
\end{proof}
\subsection{Approximation of the norm on the weighted space}
In this section, we introduce an approximation of the weight.
\begin{lemma}
\label{A}
There exist constants $c>0$ and $h_{0}>0$ such that, for all $\psi \in  \textit{$L^{2}(\hat{a}\mathrm{~d} \tau \mathrm{d} s)$} $ and $h\in(0,h_{0})\,,$ we have:
\begin{equation}
    \label{z1}
    \|\psi\|^{2}_{\textit{$L^{2}(\hat{a}^{-1}\mathrm{~d} \tau \mathrm{d} s)$}}  \leq (1+ch^{1/2}\,)\,\|\Pi _{s}\psi\|^{2}_{\textit{$L^{2}(\hat{a}\mathrm{~d} \tau \mathrm{d} s)$}}+(1+ch^{1/4}\,)\,\|\Pi^{\perp}_{s}\psi\|^{2}_{\textit{$L^{2}(\hat{a}\mathrm{~d} \tau \mathrm{d} s)$}}\,,
\end{equation}
and
\begin{equation}
    \label{A2}
    \|\psi\|^{2}_{\textit{$L^{2}(\hat{a}^{-1}\mathrm{~d} \tau \mathrm{d} s)$}}\geq (1-ch^{1/2}\,)\,\|\Pi \psi\|^{2}_{\textit{$L^{2}(\hat{a}\mathrm{~d} \tau \mathrm{d} s)$}}+(1-ch^{1/4}\,)\,\|\Pi^{\perp}_{s}\psi\|^{2}_{\textit{$L^{2}(\hat{a}\mathrm{~d} \tau \mathrm{d} s)$}}\,.
\end{equation}
\end{lemma}
\begin{proof}
We have
\begin{align*}
\bigg|\displaystyle{\int^{L}_{-L}}{\int^{h^{-\rho}}_{0}}|\psi|^{2}\hat{a}\mathrm{~d} \tau \mathrm{d} s -\displaystyle{\int^{L}_{-L}}{\int^{h^{-\rho}}_{0}}|\psi|^{2}\hat{a}^{-1}\mathrm{~d} \tau \mathrm{d} s\bigg| \leq C \displaystyle{\int^{L}_{-L}}{\int^{h^{-\rho}}_{0}}h^{1/2} \tau|\psi|^{2}\hat{a}\mathrm{~d} \tau \mathrm{d} s\,.
\end{align*}
Then, we use an orthogonal decomposition to get
\begin{align*}
&\bigg|\displaystyle{\int^{L}_{-L}}{\int^{h^{-\rho}}_{0}}|\psi|^{2}\hat{a}\mathrm{~d} \tau \mathrm{d} s-\displaystyle{\int^{L}_{-L}}{\int^{h^{-\rho}}_{0}}|\psi|^{2}\hat{a}^{-1}\mathrm{~d} \tau \mathrm{d} s\bigg|\\ & \leq  C \bigg[\displaystyle{\int^{L}_{-L}}{\int^{h^{-\rho}}_{0}}h^{1/2}\tau|\Pi_{s}\psi|^{2}\hat{a}\mathrm{~d} \tau \mathrm{d} s+ \displaystyle{\int^{L}_{-L}}{\int^{h^{-\rho}}_{0}}h^{1/2} \tau|\Pi^{\perp}_{s}\psi|^{2}\hat{a}\mathrm{~d} \tau \mathrm{d} s\bigg] \,.
\end{align*}
Thanks to \eqref{L5}, we get
$$\displaystyle{\int^{L}_{-L}}{\int^{h^{-\rho}}_{0}}h^{1/2} \tau|\Pi_{s}\psi|^{2}\hat{a}\mathrm{~d} \tau \mathrm{d} s =h^{1/2} \displaystyle{\int^{L}_{-L}}|\varphi(s)|^{2}\bigg[{\int^{h^{-\rho}}_{0}}\tau|v_{\kappa(s),h}|^{2}\hat{a}\mathrm{~d} \tau\bigg] \mathrm{d} s\leq  c \,h^{1/2} \,\|\Pi_{s} \psi\|^{2}_{\textit{$L^{2}(\hat{a}\mathrm{~d} \tau \mathrm{d} s)$}}\,.$$
We use that $h^{1/2-\rho} \leq h^{1/4}$, we obtain:
$$\displaystyle{\int^{L}_{-L}}{\int^{h^{-\rho}}_{0}}h^{1/2} \tau|\Pi^{\perp}_{s}\psi|^{2}\hat{a}\mathrm{~d} \tau \mathrm{d} s \leq h^{1/2-\rho} \displaystyle{\int^{L}_{-L}}{\int^{h^{-\rho}}_{0}} |\Pi^{\perp}_{s}\psi|^{2}\hat{a}\mathrm{~d} \tau \mathrm{d} s\leq h^{1/4}\,\|\Pi^{\perp} _{s}\psi\|^{2}_{\textit{$L^{2}(\hat{a}\mathrm{~d} \tau \mathrm{d} s)$}}\,. $$ 
\end{proof}
The following corollary is a direct consequence of \eqref{z1}.
 \begin{corollary}
\label{AZ} 
 There exist constants $c>0$ and $h_{0}>0$ such that, for all $\psi \in  \textit{$L^{2}(\hat{a}\mathrm{~d} \tau \mathrm{d} s)$} $ and $h\in(0,h_{0})\,,$ we have:
$$\|\Pi_{s}\psi\|^{2}_{\textit{$L^{2}(\hat{a}^{-1}\mathrm{~d} \tau \mathrm{d} s)$}}\leq (1+ch^{1/2}\,)\,\displaystyle{\int^{L}_{-L}}|\varphi(s)|^{2}\mathrm{~d} s\,,$$
 with $\varphi=\langle\psi,v_{\kappa(s),h}\rangle_{\textit{$L^{2}((0,h^{-\rho}),\hat{a}\mathrm{~d} \tau)$}}\,.$
\end{corollary}
\subsection{Upper bound.}\label{M1}
To separate the variables, we consider the function of the form :
\begin{equation}
    \label{kiki}
     \psi(s,\tau)=\varphi(s)\cdot v_{\kappa(s),h}(\tau)\,, 
\end{equation}
   with $\varphi \in \textit{$H^{1}(\mathbb{R}/2L\mathbb{Z})$}\,.$
   \\
   
   The following proposition provides an upper bound of the quadratic form on a subspace.
  \begin{proposition}
  \label{A1}
    Let  $\rho \in (0,1/4)\,,$ $ 0 \leq \eta <\frac{3}{2}  \,,$ $\alpha>0\,$ and $0< c_{1}<c_{2}$\, . There exist constants $c>0\,,$  $h_{0}\,>0\,$ such that, for all $h\in(0,h_{0})\,$ and $b$ satisfying 
 $$c_{1}h^{\frac{-\eta}{2}}\leq b \leq c_{2}h^{\frac{-\eta}{2}}\,, $$ we have for all $\psi$ define in \eqref{kiki}
       $$ \textit{$\widehat{q}^{\,b,\rho}_{h}$}(\psi) \leq  q_{h,\beta_{0}}^{\mathrm{eff},+}(\varphi)+\mathcal{O}(h^{2-\alpha-\eta})\|\varphi\|_{\textit{$L^{2}(\mathbb{R}/2L\mathbb{Z})$}}^{2}\,,$$
  where, for all $ \varphi \in \textit{$H^{1}(\mathbb{R}/2L\mathbb{Z})$},$
      $$q_{h,\beta_{0}}^{\mathrm{eff},+}(\varphi)=\displaystyle{\int^{L}_{-L}}\Big(-1-\kappa(s) h^{\frac{1}{2}}-\frac{\kappa(s)^{2}}{2}h \Big)\,|\varphi(s)|^{2}\mathrm{~d} s+h\big[1+c\,h^{\min(\alpha,\frac{1}{2})}\big]\displaystyle{\int^{L}_{-L}}|\,(\partial_{s}-ib\beta_{0})\varphi |^{2}\mathrm{~d} s\,.$$

 \end{proposition}
\begin{proof}
The proof will be done in a five steps.\\
        i.  We write
    $$\textit{$\widehat{q}^{\,b,\rho}_{h}$}(\psi)=\textit{${q}^{\,tang}$}(\psi)+
   \textit{${q}^{\,trans}$}(\psi)\,,$$
   where $$\,\,\,\,\,\,\,\,\,\, \textit{${q}^{\,trans}$}(\psi)=\displaystyle{\int^{L}_{-L}\int^{h^{-\rho}}_{0}}\,\bigg|\,\bigg(h^{\frac{1}{2}}\partial_{s}-ibh^{\frac{1}{2}}\beta_{0}-ibh\Big(-\tau+h^{\frac{1}{2}}\frac{\tau^{2}}{2}\kappa(s)\Big)\bigg)\psi\bigg|^{2}\,\hat{a}^{-1}\mathrm{~d} \tau \mathrm{d} s\,,$$ and
   $$\textit{${q}^{\,tang}$}(\psi)=\displaystyle{\int^{L}_{-L}\int^{h^{-\rho}}_{0}}|\partial_{\tau}\psi|^{2}\,\hat{a}\mathrm{~d} \tau \mathrm{d} s-\displaystyle{\int^{L}_{-L}}|\psi(s,\tau=0)|^{2}\mathrm{~d} s\,.$$
   ii. Upper bound of $\textit{${q}^{\,tang}$}(\psi).$  We get, by using the min-max principle and \eqref{4.4}
      \begin{align*}
   \textit{${q}^{\,tang}$}(\psi)\,&=\displaystyle{\int^{L}_{-L}\,|\varphi(s)|^{2}\bigg[\,\int^{h^{-\rho}}_{0}}|\partial_{\tau}v_{\kappa(s),h}(\tau)|^{2}\,\hat{a}\mathrm{~d} \tau -|v_{\kappa(s),h}(0)|^{2}\,\bigg]\mathrm{~d} s\\&= \displaystyle{\int^{L}_{-L}}\textit{$ q_{\kappa(s),h} $}(v_{\kappa(s),h})\,|\varphi(s)|^{2} \mathrm{~d} s\\&= \displaystyle{\int^{L}_{-L}}\lambda_{1}(\mathcal{H}_{\kappa(s),h})\,|\varphi(s)|^{2}\mathrm{~d} s\\&\leq\displaystyle{\int^{L}_{-L}}\Big(-1-\kappa(s) h^{\frac{1}{2}}-\frac{\kappa(s)^{2}}{2}h \Big)\,|\varphi(s)|^{2}\mathrm{~d} s +o(h)\|\varphi\|^{2}\,. \\
   \end{align*}
   iii. Upper bound of $\textit{${q}^{\,trans}$}(\psi)$. Using that for any $a,b \in \mathbb{R}\,,$ $h>0\,$ and $\alpha>0$ such that 
   $$|a+b|^{2}\leq (1+h^{\alpha}\,)\,|a|^{2}+(1+h^{-\alpha}\,)\,|b|^{2}\,,$$
   we have 
   $$  \textit{${q}^{\,trans}$}(\psi)\leq (1+h^{{\alpha}}\,)\,\textit{${q}_{1}^{\,trans}$}(\psi)+(1+h^{-{\alpha}})\textit{${q}_{2}^{\,trans}$}(\psi)\,,$$
   where 
   $$\textit{${q}_{1}^{\,trans}$}(\psi)=\displaystyle{\int^{L}_{-L}\int^{h^{-\rho}}_{0}}\,\bigg|\,\Big(h^{\frac{1}{2}}\partial_{s}-ibh^{\frac{1}{2}}\beta_{0}\Big)\psi\bigg|^{2}\hat{a}^{-1}\mathrm{~d} \tau \mathrm{d} s\,,$$ 
   and 
$$\textit{${q}_{2}^{\,trans}$}(\psi)=b^{2} h^{2} \displaystyle{\int^{L}_{-L}\int^{h^{-\rho}}_{0}}\,\bigg|\,\Big(-\tau+h^{\frac{1}{2}}\frac{\tau^{2}}{2}\kappa(s)\Big)\psi\bigg|^{2}\hat{a}^{-1}\mathrm{~d} \tau\mathrm{d} s\,. $$
   iv. By use Lemma \ref{A}, there exists $ c> 0 $ such that
   \begin{align*}
   \textit{${q}_{2}^{\,trans}$}(\psi)\,&\leq b^{2} h^{2}(1+c\,h^{1/4}) \displaystyle{\int^{L}_{-L}|\varphi(s)|^{2}\bigg[\int^{h^{-\rho}}_{0}} \bigg|\,\Big(-\tau+h^{\frac{1}{2}}\frac{\tau^{2}}{2}\kappa(s)\Big)v_{\kappa,h}\bigg|^{2}\hat{a}\mathrm{~d} \tau \bigg]\mathrm{~d} s\,.
   \end{align*}
Using the inequality \eqref{L5}, and since the curvature is bounded, we get with a possible new constant $c,$
  $$\displaystyle{\int^{h^{-\rho}}_{0}} \bigg|\,\Big(-\tau+h^{\frac{1}{2}}\frac{\tau^{2}}{2}\kappa(s)\Big)v_{\kappa,h}\bigg|^{2}\hat{a}\mathrm{~d} \tau \leq c\,.$$
Therefore, $$(1+h^{-\alpha}\,)\,\textit{${q}_{2}^{\,trans}$}(\psi)\leq c\,b^{2}h^{2-\alpha}\|\varphi\|^{2}\leq  c h^{2-\alpha-\eta}\|\varphi\|^{2}\,.$$
      and
 $$(1+h^{-\alpha}\,)\,\textit{${q}_{2}^{\,trans}$}(\psi)=\mathcal{O}(h^{2-\alpha-\eta})\|\varphi\|^{2}\,.$$ 
 Note that:\\
\item [-] If $0 \leq \eta<\frac{5}{4}$, we choose $\alpha=\frac{1}{2}(\frac{5}{4}-\eta)\,$ then $$\dfrac{h^{2-\alpha-\eta}}{h^{\frac{3}{4}}}=h^{\frac{5}{4}-\alpha-\eta}=h^{\frac{1}{2}(\frac{5}{4}-\eta)}\underset{h \to 0^{+}}{\longrightarrow }0\,.$$
\item[-] If $\,\frac{5}{4} \leq \eta<\frac{3}{2}$, we choose $\alpha=\frac{1}{2}(\frac{3}{2}-\eta)\,$ then $$\dfrac{h^{2-\alpha-\eta}}{h^{\frac{1}{2}}}=h^{\frac{3}{2}-\alpha-\eta}=h^{\frac{1}{2}(\frac{3}{2}-\eta)}\underset{h \to 0^{+}}{\longrightarrow }0\,.$$
 v. We have
 \begin{align*}
   \textit{${q}_{1}^{\,trans}$}(\psi)\,&\leq h\,(1+h^{\alpha}\,)\,\displaystyle{\int^{L}_{-L}\int^{h^{-\rho}}_{0}}\,|\,v_{\kappa(s),h}\,(\partial_{s}-ib\beta_{0})\varphi |^{2}\,\tilde{a}^{-1}\mathrm{~d} \tau \mathrm{d} s\\&\,\,\,\,+h(\,1+h^{-\alpha}\,)\,\displaystyle{\int^{L}_{-L}\int^{h^{-\rho}}_{0}}\, |\partial_{s}v_{\kappa(s),h}|^{2}|\varphi(s)|^{2}\,\tilde{a}^{-1}\mathrm{~d} \tau\mathrm{d} s\,.
   \end{align*}
   Using Lemma \ref{A}  and Corollary \ref{AZ}, we write
   \begin{align*}
   \textit{${q}_{1}^{\,trans}$}(\psi)\,&\leq h\,(1+h^{\alpha}\,)(1+c\,h^{1/2})\,\displaystyle{\int^{L}_{-L}}|\,(\partial_{s}-ib\beta_{0})\varphi |^{2}\mathrm{~d} s\\& \,\,\,\,+h\,(1+h^{-\alpha}\,)(1+c\,h^{1/4})\,\displaystyle{\int^{L}_{-L}|\varphi(s)|^{2}\bigg[\int^{h^{-\rho}}_{0}}\, |\partial_{s}v_{\kappa(s),h}|^{2}\hat{a}\mathrm{~d} \tau\bigg]\mathrm{~d} s\,.
   \end{align*}
   From \eqref{RO}, we deduce that
   \begin{align*}
       (1+h^{\alpha})\textit{${q}_{1}^{\,trans}$}(\psi)&\leq h\,[1+c\,h^{\text{min}(\alpha,\frac{1}{2})}\,]\displaystyle{\int^{L}_{-L}}|\,(\partial_{s}-ib\beta_{0})\varphi |^{2}\mathrm{~d} s+\mathcal{O}(h^{2-\alpha})\|\varphi\|^{2}\\&\leq h\,[1+c\,h^{\text{min}(\alpha,\frac{1}{2})}\,]\displaystyle{\int^{L}_{-L}}|\,(\partial_{s}-ib\beta_{0})\varphi |^{2}\mathrm{~d} s+\mathcal{O}(h^{2-\alpha-\eta})\|\varphi\|^{2} \,.
   \end{align*}
Then the conclusion follows.
\end{proof}
\subsection{Lower bound.}\label{M2}
The following proposition provides a lower bound of the quadratic form. 
\begin{proposition}
  \label{A10} 
   Let  $\rho \in (0,1/4)\,,$ $ 0 \leq \eta < \frac32\,,$  $\alpha>0$  and  $0< c_{1}<c_{2}$\,.   There exist constants $c>0\,,$ $h_{0}>0\,$  such that, for all $h\in(0,h_{0})\,$ and $b$ satisfying 
 $$c_{1}h^{\frac{-\eta}{2}}\leq b \leq c_{2}h^{\frac{-\eta}{2}}\,, $$ we have 
   \begin{align*}
      \textit{$\widehat{q}^{\,b,\rho}_{h}$}(\psi)& \geq  q_{h,\beta_{0}}^{\mathrm{eff},-}(\varphi)+\mathcal{O}(h^{2-\alpha-\eta})\|\varphi\|_{\textit{$L^{2}(\mathbb{R}/2L\mathbb{Z})$}}^{2}-\mathcal{O}( h^{-\eta+\frac{3}{2}-\alpha})\|\Pi^{\perp}_{s}\psi\|^{2}_{\textit{$L^{2}(\hat{a}\mathrm{~d} \tau \mathrm{d} s)$}}\,, 
   \end{align*}
   where
   \begin{enumerate}
   \item[$\bullet$] $ \textit{$\widehat{q}^{\,b,\rho}_{h}$}$ define on $\mathrm{Dom}({\textit{$\widehat{q}^{\,b,\rho}_{h}$}})$ in \eqref{ay}.
    \item[$\bullet$]  $\psi \in\mathrm{Dom}({\textit{$\widehat{q}^{\,b,\rho}_{h}$}})$.
     \item[$\bullet$]  $ \varphi \in \textit{$H^{1}(\mathbb{R}/2L\mathbb{Z})$} $ defined by  $$\varphi(s):=\langle\psi(s,\cdot),v_{\kappa(s),h}\rangle_{\textit{$L^{2}((0,h^{-\rho}),\hat{a}\mathrm{~d} \tau)$}}\,,$$
  and
   $$ \begin{aligned}
        \quad q_{h,\beta_{0}}^{\mathrm{eff},-}(\varphi)&=\displaystyle{\int^{L}_{-L}}\Big(-1-\kappa(s) h^{\frac{1}{2}}-\frac{\kappa(s)^{2}}{2}h \Big)\,|\varphi(s)|^{2}\mathrm{~d} s\\&\quad\quad\quad\quad\quad\quad\quad\quad\quad\quad\quad\quad+h[1-c\,h^{\min(\alpha,\frac{1}{2})})]\displaystyle{\int^{L}_{-L}}|\,(\partial_{s}-ib\beta_{0})\varphi |^{2}\mathrm{~d} s\,.
      \end{aligned}$$
       \end{enumerate}
 \end{proposition}
 \begin{proof}
 We write 
   \begin{align*}
   \textit{$\widehat{q}^{\,b,\rho}_{h}$}(\psi)=\textit{${q}^{\,trans}$}(\psi)+\textit{${q}^{\,tang}$}(\psi)\,,
   \end{align*}
    with $$\,\,\,\,\,\,\,\,\,\,\textit{${q}^{\,trans}$}(\psi)=\displaystyle{\int^{L}_{-L}\int^{h^{-\rho}}_{0}}\,\bigg|\,\bigg(h^{\frac{1}{2}}\partial_{s}-ibh^{\frac{1}{2}}\beta_{0}-ibh\Big(-\tau+h^{\frac{1}{2}}\frac{\tau^{2}}{2}\kappa(s)\Big)\bigg)\psi\bigg|^{2}\,\hat{a}^{-1}\mathrm{~d} \tau \mathrm{d} s\,,$$ and 
   $$\textit{${q}^{\,tang}$}(\psi)=\displaystyle{\int^{L}_{-L}\int^{h^{-\rho}}_{0}}|\partial_{\tau}\psi|^{2}\hat{a}\mathrm{~d} \tau \mathrm{d}s-\displaystyle{\int^{L}_{-L}}|\psi(s,\tau=0)|^{2}\mathrm{~d} s\,.$$\\
     i. By using the orthogonal decomposition
     $\psi=\Pi_{s}\psi +\Pi^{\perp}_{s}\psi$, we get:
       $$\textit{$ q_{\kappa(s),h}$}(\psi)=\textit{$ q_{\kappa(s),h}$}(\Pi_{s}\psi )+\textit{$ q_{\kappa(s),h}$}(\Pi^{\perp}_{s}\psi)\,,$$
       Then, by using the min-max principle,we get
       \begin{align*}
         \textit{${q}^{\,tang}$}(\psi)&=\displaystyle{\int^{L}_{-L}}\textit{$ q_{\kappa(s),h}$}(\psi)\mathrm{~d} s \\&=\displaystyle{\int^{L}_{-L}} \textit{$ q_{\kappa(s),h}$}(\varphi v_{\kappa(s),h})\mathrm{~d} s+\displaystyle{\int^{L}_{-L}} \textit{$ q_{\kappa(s),h}$}(\Pi^{\perp}_{s}\psi)\mathrm{~d} s\\& \geq \displaystyle{\int^{L}_{-L}} \lambda_{1}(\mathcal{H}_{\kappa(s),h})\|\varphi v_{\kappa(s),h}\|^{2}_{\textit{$L^{2}((0,h^{-\rho}),\hat{a}\,d\tau)$}} \mathrm{~d} s \\&\quad+ \displaystyle{\int^{L}_{-L}} \lambda_{2}(\mathcal{H}_{\kappa(s),h})\|\Pi^{\perp}_{s}\psi\|^{2}_{\textit{$L^{2}((0,h^{-\rho}),\hat{a}\,d\tau)$}}\mathrm{~d} s\,.
         \end{align*}
From \eqref{4.4} and \eqref{&é}, we have 
       $$\textit{${q}^{\,tang}$}(\psi)\geq\displaystyle{\int^{L}_{-L}}\Big(-1-\kappa(s) h^{\frac{1}{2}}-\frac{\kappa(s)^{2}}{2}h \Big)\,|\varphi(s)|^{2}\mathrm{~d} s +o(h)\|\varphi\|^{2}-Ch^{1/2-\rho}\|\Pi^{\perp}_{s}\psi\|^{2}_{\textit{$L^{2}(\hat{a}\mathrm{~d} \tau \mathrm{d} s)$}}\,.$$
    ii. By Lemma \ref{A}, we have 
       \begin{align*}
          \textit{${q}^{\,trans}$}(\psi)& \geq(1-ch^{1/2}\,)\displaystyle{\int^{L}_{-L}\int^{h^{-\rho}}_{0}}\,\bigg|h^{\frac{1}{2}}\,\Pi_{s}\partial_{s}\psi-ibh^{\frac{1}{2}}\beta_{0}\Pi_{s}\psi-ibh\Pi_{s}(\alpha_{s}\psi)\bigg|^{2}\,\hat{a}\mathrm{~d} \tau \mathrm{d} s \\&\,\,\,\,\,\,\,\,\,\,\,+(1-ch^{1/4}\,)\displaystyle{\int^{L}_{-L}\int^{h^{-\rho}}_{0}}\,\bigg|h^{\frac{1}{2}}\,\Pi^{\perp}_{s}\partial_{s}\psi-ibh^{\frac{1}{2}}\beta_{0}\Pi^{\perp}_{s}\psi-ibh\Pi^{\perp}_{s}(\alpha_{s}\psi)\bigg|^{2}\,\hat{a}\mathrm{~d} \tau \mathrm{d} s\,.
      \end{align*}
     with   $$\alpha_{s}(\tau)=-\tau+h^{1/2}\frac{\tau^{2}}{2}\kappa(s)\,.$$
     We write $$\,\,\,\,\,\, \Pi_{s}\partial_{s}\psi=\partial_{s}\Pi_{s}\psi+[\Pi_{s},\partial_{s}]\psi\,,$$   and use the following classical inequality. For any $a,b \in \mathbb{R}\,,$ $h>0\,$ and $\alpha>0$, we have 
     $$|a+b|^{2} \geq (1-h^{\alpha})|a|^{2}-h^{-\alpha}|b|^{2}\,.$$
       We obtain
       \begin{align*}
           \textit{${q}^{\,trans}$}(\psi)&\geq (1-ch^{1/2})(1-h^{\alpha})\textit{${q}_{1}^{\,trans}$}(\psi)-(1-ch^{1/2})h^{-\alpha}\textit{${q}_{2}^{\,trans}$}(\psi)\\&\,\,\,\,\,+(1-ch^{1/4})(1-h^{\alpha})\textit{${q}_{3}^{\,trans}$}(\psi)-(1-ch^{1/4})h^{-\alpha}\textit{${q}_{2}^{\,trans}$}(\psi)\,,
       \end{align*}
        with \\
$\bullet$ $\textit{${q}_{1}^{\,trans}$}(\psi)=\displaystyle{\int^{L}_{-L}\int^{h^{-\rho}}_{0}}\,\Big|h^{\frac{1}{2}}\,\partial_{s}\Pi_{s}\psi-ibh^{\frac{1}{2}}\beta_{0}\Pi_{s}\psi-ibh\Pi_{s}(\alpha_{s}\psi)\Big|^{2}\,\hat{a}\mathrm{~d} \tau \mathrm{d} s\,,$\\
$\bullet$ $\textit{${q}_{2}^{\,trans}$}(\psi)=h\displaystyle{\int^{L}_{-L}\int^{h^{-\rho}}_{0}}\Big|[\Pi_{s},\partial_{s}]\psi\Big|^{2}\,\hat{a}\,d\tau ds\,=
h\displaystyle{\int^{L}_{-L}\int^{h^{-\rho}}_{0}}\Big|[\Pi^{\perp}_{s},\partial_{s}]\psi\Big|^{2}\hat{a}\mathrm{~d} \tau \mathrm{d} s\,,$\\
$\bullet$ $\textit{${q}_{3}^{\,trans}$}(\psi)=\displaystyle{\int^{L}_{-L}\int^{h^{-\rho}}_{0}}\,\Big|h^{\frac{1}{2}}\,\partial_{s}\Pi^{\perp}_{s}\psi-ibh^{\frac{1}{2}}\beta_{0}\Pi^{\perp}_{s}\psi-ibh\Pi^{\perp}_{s}(\alpha_{s}\psi)\Big|^{2}\,\hat{a}\mathrm{~d} \tau \mathrm{d} s\,.$\\
     iii. We have 
       \begin{align*}
\textit{${q}_{1}^{\,trans}$}(\psi)&\geq (1-h^{\alpha})\displaystyle{\int^{L}_{-L}\int^{h^{-\rho}}_{0}}\,\Big|h^{\frac{1}{2}}\,\partial_{s}\Pi_{s}\psi-ibh^{\frac{1}{2}}\beta_{0}\Pi_{s}\psi\Big|^{2}\,\hat{a}\mathrm{~d} \tau \mathrm{d} s\\& \\&\quad\quad\quad\quad\quad\quad\quad\quad\quad\quad-h^{-\alpha}\displaystyle{\int^{L}_{-L}\int^{h^{-\rho}}_{0}}\,\Big|bh\Pi_{s}(\alpha_{s}\psi)\Big|^{2}\hat{a}\mathrm{~d} \tau \mathrm{d} s\,.
\end{align*}
Clearly,
    \begin{align*}
    \displaystyle{\int^{L}_{-L}\int^{h^{-\rho}}_{0}}\,\Big|\Pi_{s}(\alpha_{s}\psi)\Big  |^{2}\,\hat{a}\mathrm{~d} \tau \mathrm{d} s &\leq \displaystyle{\int^{L}_{-L}\int^{h^{-\rho}}_{0}}\,\Big|\alpha_{s}\psi\Big|^{2}\hat{a}\mathrm{~d} \tau \mathrm{d} s \\& \leq 2\displaystyle{\int^{L}_{-L}\int^{h^{-\rho}}_{0}}\,\Big|\alpha_{s}\Pi_{s}\psi\Big|^{2}\hat{a}\mathrm{~d} \tau \mathrm{d} s+2\displaystyle{\int^{L}_{-L}\int^{h^{-\rho}}_{0}}\,\Big|\alpha_{s}\Pi^{\perp}_{s}\psi\Big|^{2}\hat{a}\mathrm{~d} \tau \mathrm{d} s \,.
     \end{align*}
With the same type of reasoning as for the upper bound, there exists a constant $c>0$ such that 
\begin{align*}
    \textit{${q}_{1}^{\,trans}$}(\psi)&\geq h(1-c\,h^{\alpha})\displaystyle{\int^{L}_{-L}}|\,(\partial_{s}-ib\beta_{0})\varphi |^{2}\mathrm{~d} s+\mathcal{O}(h^{2-\alpha-\eta})\|\varphi\|_{\textit{$L^{2}(\mathbb{R}/2L\mathbb{Z})$}}^{2}\\& \quad -\mathcal{O}(b^{2} h^{\frac{3}{2}-\alpha})\|\Pi^{\perp}_{s}\psi\|^{2}_{\textit{$L^{2}(\hat{a}\mathrm{~d} \tau \mathrm{d} s)$}}\,.
\end{align*}
   vi. We use Lemma \ref{L1}  and the inequality \eqref{RO}, to obtain
       $$
        \textit{${q}_{2}^{\,trans}$}(\psi)\leq c h^{2} \|\psi\|_{\textit{$L^{2}(\hat{a}\mathrm{~d} \tau\mathrm{d} s)$}}\,\leq  ch^{2}\|\varphi\|^{2}+ch^{2}\|\Pi^{\perp}_{s}\psi\|^{2}_{\textit{$L^{2}(\hat{a}\mathrm{~d} \tau \mathrm{d} s)$}}\,.
       $$
      v.
       Since $\textit{${q}_{3}^{\,trans}$}(\psi)\geq 0\,, $ we obtain:
       \begin{align*}
             \textit{${q}^{\,trans}$}(\psi) & \geq h(1-c\,h^{\text{min}(\alpha,\frac{1}{2})})\displaystyle{\int^{L}_{-L}}|\,(\partial_{s}-ib\gamma_{0})\varphi |^{2}\mathrm{~d} s + \mathcal{O}(h^{2-\alpha-\eta})\,\|\varphi\|^{2} \\& \quad -ch^{-\eta+3/2-\alpha}\,\|\Pi^{\perp}_{s}\psi\|^{2}_{\textit{$L^{2}(\hat{a}\mathrm{~d} \tau \mathrm{~d} s)$}}\ \,.
       \end{align*}
   Then the conclusion of the proposition follows.
 \end{proof}
   \subsection{End of the proof}
   We now have everything to finish the proof of Theorem \ref{R21}.
 The self-adjoint operator associated with the quadratic form
 $q_{h,\beta_{0}}^{\mathrm{eff},\pm}$ is :$$\mathcal{L}_{h,\beta_{0}}^{\mathrm{eff},\pm}=-h(1\pm c_{\pm}\,h^{\text{min}(\alpha,\frac{1}{2})})(\partial_{s}-ib\beta_{0})^{2}-1-\kappa(s) h^{\frac{1}{2}}-\frac{\kappa(s)^{2}}{2}h$$
  in  $\textit{$L^{2}(\mathbb{R}/2L\mathbb{Z})$}\,,$ where $c_{\pm}$ is a constant independent of $h$. The operator  $\mathcal{L}_{h,\beta_{0}}^{\mathrm{eff},\pm}$ is with compact resolvent and it is bounded from below. Its spectrum is purely discrete and it is consists of by an increasing sequence of eigenvalues $\lambda_{n}(\mathcal{L}_{h,\beta_{0}}^{\mathrm{eff},\pm})\,.$
   \begin{corollary}
   \label{11a}
     Let $n \in \mathbb{N}^{*},\,$ $\rho \in (0, 1/4) \,,$ $ 0 \leq \eta <\frac{3}{2}  \,,$ $\alpha>0\,$ and $0< c_{1}<c_{2}$\,. There exist  $h_{0}\,>0,\,$ such that, if $h\in(0,h_{0})\,$
  and $b$ satisfies
 $$c_{1}h^{\frac{-\eta}{2}}\leq b \leq c_{2}h^{\frac{-\eta}{2}}\,, $$  we have :
   $$\widehat{\mu}_{n}(h,b,\rho) \leq \lambda_{n}(\mathcal{L}_{h,\beta_{0}}^{\mathrm{eff},+})+\mathcal{O}(h^{2-\alpha-\eta})\,.$$
   \end{corollary}
   \begin{proof}
   Let $(\varphi_{j})_{ 1\leq j\leq n}$ an orthonormal family of eigenvectors associated with eigenvalues\\ $(\lambda_{j}(\mathcal{L}_{h,\beta_{0}}^{\mathrm{eff},+}))_{1\leq j\leq n}\,.$
   Let $E$ the subspace of $\mathrm{Dom}(q_{h,\beta_{0}}^{\mathrm{eff},+})$ generated by the family $(\varphi_{j})_{ 1\leq j\leq n}\,.$
By Proposition  \ref{A1}, we have:  
    $$\textit{$\widehat{q}^{\,b,\rho}_{h}$}(v_{k(s),h}\,\varphi)\leq q_{h,\beta_{0}}^{\mathrm{eff},+}(\varphi)+\mathcal{O}(h^{2-\alpha-\eta})\|\varphi\|^{2}\,,\quad \quad \forall \varphi \in E\,.$$
 We use the min-max theorem and deduce
  $$q_{h,\beta_{0}}^{\mathrm{eff},+}(\varphi)\leq \lambda_{n}(\mathcal{L}_{h,\beta_{0}}^{\mathrm{eff},+})\|\varphi\|^{2} \,.$$  
  Consequently,
   \begin{align*}
    \textit{$\widehat{q}^{\,b,\rho}_{h}$}(v_{k(s),h}\,\varphi)&\leq  \lambda_{n}(\mathcal{L}_{h,\beta_{0}}^{\mathrm{eff},+})\|\varphi\|^{2}+\mathcal{O}(h^{2-\alpha-\eta})\|\varphi\|^{2}\\& \leq\Big( \lambda_{n}(\mathcal{L}_{h,\beta_{0}}^{\mathrm{eff},+})+\mathcal{O}(h^{2-\alpha-\eta})\Big)\,\|v_{k(s),h} \varphi\|^{2}\,.
   \end{align*}
 For all  $g\in v_{k(s),h}E\,,$ we have:
   $$\textit{$\widehat{q}^{\,b,\rho}_{h}$}(g)\leq\Big( \lambda_{n}(\mathcal{L}_{h,\beta_{0}}^{\mathrm{eff},+})+\mathcal{O}(h^{2-\alpha-\eta})\Big)\,\|g\|^{2}\,.$$
   Since $$\displaystyle{\int_{0}^{h^{-\rho}} }|v_{k(s),h}(\tau)|^{2}\hat{a}\mathrm{~d} \tau= 1\,,$$ 
 then the functions of the family $ (v_{k(s), h} \varphi_{j})_{1 \leq j \leq n} $ are linearly independent, so  $\mathrm{dim}(f_{k(s),h}E)=n\,.$ By application of the min-max principle, we obtain  $$ \widehat{\mu}_{n}(h,b,\rho)\leq  \lambda_{n}(\mathcal{L}_{h,\beta_{0}}^{\mathrm{eff},+})+\mathcal{O}(h^{2-\alpha-\eta}).$$
\end{proof} 
 \begin{corollary}
 \label{11b}
Let $n \in \mathbb{N}^{*},\,$ $\rho \in (0, 1/4) \,,$ $ 0 \leq \eta <\frac{3}{2}  \,,$ $\alpha>0\,$ and $0< c_{1}<c_{2}$\,. There exist  $h_{0}\,>0,\,$ such that, if $h\in(0,h_{0})\,$
  and $b$ satisfies
 $$c_{1}h^{\frac{-\eta}{2}}\leq b \leq c_{2}h^{\frac{-\eta}{2}}\,, $$   then we have :   
   $$\widehat{\mu}_{n}(h,b,\rho) \geq \lambda_{n}(\mathcal{L}_{h,\beta_{0}}^{\mathrm{eff},-})+\mathcal{O}(h^{2-\alpha-\eta})\,. $$
   \end{corollary}
   \begin{proof}
If $\eta< \frac32 $, there exist  $0<\alpha\leq \frac32-\eta$ then $-\eta+\frac{3}{2}-\alpha\geq 0 \,.$ According to Proposition \ref{A10}, there exist $\epsilon_{0}\in (0,1)$ such that, for all $\psi \in\mathrm{Dom}({\textit{$\widehat{q}^{\,b,\rho}_{h}$}})$ and $h$ small enough, we have:
   \begin{equation}
   \label{25}
       \textit{$\widehat{q}^{\,b,\rho}_{h}$}(\psi) \geq  Q_{h}^{tens}(\langle\psi,v_{\kappa(s),h} \rangle,\Pi^{\perp}\psi)\,,
   \end{equation}
   where, $\forall (\varphi,f)\in\mathrm{Dom}(q_{h,\beta_{0}}^{\mathrm{eff},-}) \times\mathrm{Dom}(\textit{$\widehat{q}^{b,\rho}_{h} $})$
    $$Q_{h}^{tens}(\varphi,f) =Q_{h}^{\text{eff}}(\varphi)-\dfrac{\epsilon_{0}}{2}\|f\|^{2}_{\textit{$L^{2}(\tilde{a}\mathrm{~d} \tau \mathrm{d} s)$}}\,,$$
     with $$ Q_{h}^{\text{eff}}(\varphi)=q_{h,\beta_{0}}^{\mathrm{eff},-}(\varphi)+\mathcal{O}(h^{2-\alpha-\eta}))\|\varphi\|_{\textit{$L^{2}(\mathbb{R}/2L\mathbb{Z})$}}^{2}\,.$$
   By application of the min-max principle (see \cite[Chapter 13]{6}), we have the comparison of the Rayleigh quotients:
    \begin{equation}
        \label{26}
  \widehat{\mu}_{n}(h,b,\rho) \geq  \widehat{\mu}_{n}^{tens}(h) \,.
    \end{equation}

We note that the self-adjoint operators associated with the quadratic forms $Q_{h}^{\mathrm{eff}}$  and $Q_{h}^{tens}$ are respectively $\mathcal{L}_{h}^{\text{eff}}\,$  and $\mathcal{L}_{h}^{tens}\,.$
 It is easy to see that the spectrum  of $\mathcal{L}_{h}^{tens}$ lying below $-\epsilon_{0}$ coincides with the spectrum of $\mathcal{L}_{h}^{\text{eff}}\,.$ Then, for all $n\in \mathbb{N}^{*}\,,$  $\widehat{\mu}_{n}^{tens}(h)$ and $\widehat{\mu}_{n}^{\text{eff}}(h)$
 are respectively the $n$-th eigenvalues of  $\mathcal{L}_{h}^{tens}$ and
 $\mathcal{L}_{h}^{\text{eff}}$ and satisfy
\begin{equation}
    \label{27}
    \widehat{\mu}_{n}^{tens}(h)=\widehat{\mu}_{n}^{\text{eff}}(h)   =\lambda_{n}(\mathcal{L}_{h,\beta_{0}}^{\mathrm{eff},-})+\mathcal{O}(h^{2-\alpha-\eta})\,.
\end{equation}
\end{proof}
\section{Spectrum of the effective operator}
\label{105}
Let $\lambda_{n}(\mathcal{L})$ be the $n$-th eigenvalue of an operator $\mathcal{L}\,.$ Fix a constant $c \in \mathbb{R}\,.$ The  effective operators $\mathcal{L}_{h,\beta_{0}}^{\mathrm{eff},\pm}$ are special cases of the following operator: 
$$\mathcal{L}_{h,\alpha,0}^{\mathrm{eff},b}=-h(1-c\,h^{\text{min}(\alpha,\frac{1}{2})})(\partial_{s}-ib\beta_{0})^{2}-1-\kappa(s) h^{\frac{1}{2}}-\frac{\kappa(s)^{2}}{2}h\,,$$
with $b>0$  and $\alpha>0\,.$  The associate quadratic form on $\textit{$L^{2}(\mathbb{R}/2L\mathbb{Z})$}$  is $$q_{h,\alpha,0}^{\mathrm{eff},b}(\varphi)=\displaystyle{\int^{L}_{-L}}\Big(-1-\kappa(s) h^{\frac{1}{2}}-\frac{\kappa(s)^{2}}{2}h \Big)\,|\varphi(s)|^{2}\mathrm{~d}s+h(1-c\,h^{\text{min}(\alpha,\frac{1}{2})})\displaystyle{\int^{L}_{-L}}|\,(\partial_{s}-ib\beta_{0})\varphi |^{2}\mathrm{~d}s\,.$$
In the following, we will explain how the spectrum of the operator $\mathcal{L}_{h,\alpha,0}^{\mathrm{eff},b}$ reduces to the study of that stated in Theorem \ref{R1}. The calculations in this section are classical. The operator that we are going to study its spectrum is a perturbation of the harmonic oscillator. Then, the calculations are based on the context of the harmonic approximation (cf.\cite[Section 4.3]{6}), with a small adjustment, since the operator is defined on a periodic space and with a phase term $ib\beta_{0}$.\\
Let $$\mathcal{L}_{h,\alpha,1}^{\mathrm{eff},b}=\mathcal{L}_{h,\alpha,0}^{\mathrm{eff},b}+1.$$
By definition of the spectrum, we have
\begin{equation}
\label{30}
    \lambda_{n}(\mathcal{L}_{h,\alpha,0}^{\mathrm{eff},b})=-1+\lambda_{n}(\mathcal{L}_{h,\alpha,1}^{\mathrm{eff},b})\,.
\end{equation}
Factoring the term $h^{-1/2}$, we get the new operator
$$\mathcal{L}_{h,\alpha,2}^{\mathrm{eff},b}=h^{-1/2}\mathcal{L}_{h,\alpha,1}^{\mathrm{eff},b}\,,$$
and
\begin{equation}
      \label{31}
       \lambda_{n}(\mathcal{L}_{h,\alpha,1}^{\mathrm{eff},b})=h^{1/2} \lambda_{n}(
  \mathcal{L}_{h,\alpha,2}^{\mathrm{eff},b})\,.
  \end{equation}
We introduce the  semi-classical parameter
  $$\hbar = h^{1/4}\,.$$

 \begin{lemma}
   \label{1²} 
Let us consider:
$$\mathcal{L}_{h,\alpha,3}^{\mathrm{eff},b}=-\hbar^{2}(1-c\,\hbar^{\min(4\alpha,2)})(\partial_{s}-ib\beta_{0})^{2}-\kappa(s)\,,$$
with the associated quadratic form on $\textit{$L^{2}(\mathbb{R}/2L\mathbb{Z})$}$, $$q_{h,\alpha,3}^{\mathrm{eff},b}(\varphi)=-\displaystyle{\int^{L}_{-L}}\kappa(s)\,|\varphi(s)|^{2}\mathrm{~d} s+\hbar^{2}(1-c\,\hbar^{\min(4\alpha,2)})\displaystyle{\int^{L}_{-L}}|\,(\partial_{s}-ib\beta_{0})\varphi |^{2}\mathrm{~d} s\,.$$
Then, we have
\begin{equation}
\label{312}
    \lambda_{n}(\mathcal{L}_{h,\alpha,2}^{\mathrm{eff},b})=\lambda_{n}(\mathcal{L}_{h,\alpha,3}^{\mathrm{eff},b})+\mathcal{O}(\hbar^{2})\,.
\end{equation}
\end{lemma}
\begin{proof}
Let $\varphi \in \mathrm{H}^{1}(\mathbb{R}/2L\mathbb{Z})$. We have
$$\bigg|q_{h,\alpha,2}^{\mathrm{eff},b}(\varphi)-q_{h,\alpha,3}^{\mathrm{eff},b}(\varphi)\bigg|=\hbar^{2} \bigg| \displaystyle{\int^{L}_{-L}}\frac{\kappa(s)^{2}}{2} \,|\varphi(s)|^{2}\,ds\bigg|\,.$$
Since  $\Omega$ is a bounded domain with a smooth boundary,  the curvature $\kappa$ is also bounded, hence there exists $ c>0 $ independent of $ s $, such that
$$\bigg|q_{h,\alpha,2}^{\mathrm{eff},b}(\varphi)-q_{h,\alpha,3}^{\mathrm{eff},b}(\varphi)\bigg|\leq c \,\hbar^{2} \|\varphi\|^{2}\,.$$
Applying the min-max principle, finishes the proof.
\end{proof}
From  \eqref{30}, \eqref{31} and \eqref{312}, we obtain
\begin{equation}
    \label{331}
    \lambda_{n}(\mathcal{L}_{h,\alpha,0}^{\mathrm{eff},b})=-1+\hbar^{2}\lambda_{n}(\mathcal{L}_{h,\alpha,3}^{\mathrm{eff},b})+\mathcal{O}(\hbar^{4})\,.
\end{equation}
Therefore, to study the spectrum of the effective operator $\mathcal{L}_{h,\alpha,0}^{\mathrm{eff},b}\,,$ it suffices to study the spectrum of the operator $\mathcal{L}_{h,\alpha,3}^{\mathrm{eff},b}\,,$ introduced in Lemma \ref{1²}.
\subsection{Localization near the point of maximum curvature}
Let $\kappa_{\text{max}}$ be the maximum curvature along the boundary $\Gamma\,$ is introduced in \eqref{NR}.
Consider the effective operator of the form:
$$\mathcal{L}_{h,\alpha,4}^{\mathrm{eff},b}=-\hbar^{2}(1-c\,\hbar^{\text{min}(4\alpha,2)})(\partial_{s}-ib\beta_{0})^{2}+\kappa_{\text{max}}-\kappa(s)\,,$$
and the quadratic form on $\textit{$L^{2}(\mathbb{R}/2L\mathbb{Z})$}$ associated to it,$$q_{h,\alpha,4}^{\mathrm{eff},b}(\varphi)=\displaystyle{\int^{L}_{-L}}\big(\kappa_{\text{max}}-\kappa(s) \big)\,|\varphi(s)|^{2}\mathrm{~d}s+\hbar^{2}(1-c\,\hbar^{\min(4\alpha,2)})\displaystyle{\int^{L}_{-L}}|\,(\partial_{s}-ib\beta_{0})\varphi |^{2}\mathrm{~d}s\,.$$
Notice that:
\begin{equation}
\label{32}
    \lambda_{n}(\mathcal{L}_{h,\alpha,3}^{\mathrm{eff},b})=\lambda_{n}(\mathcal{L}_{h,\alpha,4}^{\mathrm{eff},b})-\kappa_{\text{max}}\,.
\end{equation}
\begin{proposition}
\label{yeye}
Let us consider  the effective operator of the form
$$\mathcal{L}_{h,4}^{\mathrm{eff},b}=-\hbar^{2}(\partial_{s}-ib\beta_{0})^{2}+\kappa_{\text{max}}-\kappa(s)\,,$$
and the quadratic form on $\textit{$L^{2}(\mathbb{R}/2L\mathbb{Z})$}$ associated to it, $$q_{h,4}^{\mathrm{eff},b}(\varphi)=\displaystyle{\int^{L}_{-L}}\big(\kappa_{\text{max}}-\kappa(s) \big)\,|\varphi(s)|^{2}\mathrm{~d}s+\hbar^{2}\displaystyle{\int^{L}_{-L}}|\,(\partial_{s}-ib\beta_{0})\varphi |^{2}\mathrm{~d}s\,.$$
Then, we have $$ \lambda_{n}(\mathcal{L}_{h,\alpha,4}^{\mathrm{eff},b})= (1+\mathcal{O}(\,\hbar^{\min(4\alpha,2)})\lambda_{n}(\mathcal{L}_{h,4}^{\mathrm{eff},b})\,. $$
\end{proposition}
\begin{proof}
Let $\varphi \in \mathrm{H}^{1}(\mathbb{R}/2L\mathbb{Z})$. We have:
\begin{align*}
   | q_{h,\alpha,4}^{\mathrm{eff},b}(\varphi)-q_{h,4}^{\mathrm{eff},b}(\varphi)|&=|c| \hbar^{\min(4\alpha,2)} \hbar^{2}\,\displaystyle{\int^{L}_{-L}}|\,(\partial_{s}-ib\beta_{0})\varphi |^{2}\mathrm{~d}s\\&\leq |c| \hbar^{\min(4\alpha,2)} \bigg[\hbar^{2}\,\displaystyle{\int^{L}_{-L}}|\,(\partial_{s}-ib\beta_{0})\varphi |^{2}\mathrm{~d}s+\displaystyle{\int^{L}_{-L}}\big(\kappa_{\text{max}}-\kappa(s) \big)\,|\varphi(s)|^{2}\mathrm{~d}s\bigg]\\&\leq |c| \hbar^{\min(4\alpha,2)} q_{h,4}^{\mathrm{eff},b}(\varphi)\,.
\end{align*}
  The conclusion of the lemma is now a simple application of the min-max principle.
\end{proof}
From \eqref{331}, \eqref{32} and Proposition \ref{yeye}, we obtain
$$ \lambda_{n}(\mathcal{L}_{h,\alpha,0}^{\mathrm{eff},b})=-1-\hbar^{2} \kappa_{\text{max}}+\hbar^{2}(1+\mathcal{O}(\,\hbar^{\min(4\alpha,2)})\lambda_{n}(\mathcal{L}_{h,4}^{\mathrm{eff,b}})+\mathcal{O}(\hbar^{4})\,.$$
In the following, we  show that $\lambda_{n}(\mathcal{L}_{h,4}^{\mathrm{eff},b})=\mathcal{O}(\hbar)\,$ then, we get ( recall that $\hbar=h^{1/4}$)
\begin{equation}
    \label{jiu}
    \lambda_{n}(\mathcal{L}_{h,\alpha,0}^{\mathrm{eff},b})=-1-\hbar^{2} \kappa_{\text{max}}+\hbar^{2}\lambda_{n}(\mathcal{L}_{h,4}^{\mathrm{eff},b})+o(\hbar^{3})=\lambda_{n}(\mathcal{L}_{h,\beta_{0}}^{\mathrm{eff}})+\tilde{\mathcal{O}}(h^{2-\eta})\,.
\end{equation}
Notice that Theorem \ref{R21} and the equality \eqref{jiu}, yields the proof of Theorem \ref{R1} which is a reformulation of Theorem \ref{Q2}.
\\

The goal is to show that the eigenfunctions of $\mathcal{L}_{h,4}^{\mathrm{eff},b}\,,$ associated with the eigenvalues of order $ \hbar $ concentrate exponentially near the point of maximum curvature.
\begin{proposition}
  \label{123}
For  $\epsilon \in (0,1)\,,$ there exist  $C>0$ and  $\hbar_{0}>0$ such that, for all $\hbar\in (0,\hbar_{0})$ and for all eigenfunctions $\psi$ of $\mathcal{L}_{h,,4}^{\mathrm{eff},b}\,,$ associated with the eigenvalues of orders $ \hbar \,, $ we have:
   $$ \|e^{\epsilon \phi_{0}/\hbar}\,\psi\|^{2}\leq C \|\psi\|^{2}\,\text{,}\,\,\,\,\,\,\,\,\, q_{h,4}^{\mathrm{eff},b}(e^{\epsilon \phi_{0}/\hbar}\,\psi) \leq C \hbar \|\psi\|^{2}\,, $$
   where \[\phi_{0}(s)=\min_{k\in \mathbb{Z}}\,\biggl|\displaystyle{\int^{s+2kL}_{0}} \sqrt{\kappa_{\text{max}}-\kappa(y)}  \mathrm{~d}y\biggr|\,.\]
\end{proposition}
\begin{proof}
The proof consists of three steps. Let $\epsilon \in (0,1)\,.$ Consider an eigenvalue
 $\lambda=\mathcal{O}(\hbar)$ and an associated eigenfunction $\psi\,.$ \\ \\
\textbf{1. First step.} 
 \\ \\
 Let $\phi$ be a regular function, $2L$ periodic, with real value. We have:
\begin{align*}
    \mathcal{L}_{h,4}^{\mathrm{eff},b}(e^{ \phi}\,\psi)& =e^{ \phi} \mathcal{L}_{h,4}^{\mathrm{eff},b}(\psi)-\hbar^{2}\bigg( \psi\partial_{s}^{2}e^{\phi}+2\partial_{s}e^{\phi}\partial_{s}\psi-2ib\beta_{0}\psi\partial_{s}e^{\phi}\bigg)\,.\end{align*}
Then\begin{align*}
     q_{h,4}^{\mathrm{eff},b}(e^{ \phi}\psi)&=\langle\mathcal{L}_{h,4}^{\mathrm{eff},b}(e^{ \phi}\psi),e^{ \phi}\psi \rangle\\&=\langle e^{ \phi}\mathcal{L}_{h,4}^{\mathrm{eff,b}}(\psi),e^{ \phi}\psi \rangle-\hbar^{2}\Re\Bigg[\displaystyle{\int_{-L}^{L}}\bigg[ \psi\partial_{s}^{2}e^{\phi}+2\partial_{s}e^{\phi}\partial_{s}\psi-2ib\beta_{0}\psi\partial_{s}e^{\phi} \bigg] \bar{\psi}e^{\phi}\mathrm{~d}s\Bigg]\\ & =\lambda\|e^{\phi}\psi\|^{2}-\hbar^{2}\Re\Bigg[\displaystyle{\int_{-L}^{L}}\bigg[ \psi\partial_{s}^{2}e^{\phi}+2\partial_{s}e^{\phi}\partial_{s}\psi-2ib\beta_{0}\psi\partial_{s}e^{\phi} \bigg] \bar{\psi}e^{\phi}\mathrm{~d}s\Bigg]\,.
\end{align*}
By the formula of integration by parts and as $\phi$ is $2L$ periodic,  we have 
$$\Re\displaystyle{\int_{-L}^{L}}\bigg[ \psi\partial_{s}^{2}e^{\phi}+2\partial_{s}e^{\phi}\partial_{s}\psi-2ib\beta_{0}\psi\partial_{s}e^{\phi} \bigg] \bar{\psi}e^{\phi}\mathrm{~d}s=-\|\partial_{s}e^{\phi}\psi\|^{2}\,.$$
Then, we get 
\begin{equation}
    \label{33}
     q_{h,4}^{\mathrm{eff,b}}(e^{ \phi}\,\psi)= \lambda \|e^{ \phi}\,\psi\|^{2}+\hbar^{2}\|\partial_{s}e^{ \phi}\,\psi\|^{2}\,.
\end{equation}
\textbf{2. Second step.} 
\\

For  $\,\phi=\dfrac{\epsilon \phi_{0}}{\hbar}\,,$ we notice that:
\begin{equation}
    \label{41}
 |\phi^{'}|^{2}=\dfrac{\epsilon^{2} }{\hbar^{2}}|\phi_{0}^{'}|^{2}=\dfrac{\epsilon^{2} }{\hbar^{2}} \big(\kappa_{\text{max}}-\kappa(s)\big)\,. 
\end{equation}
We have
$$q_{h,4}^{\mathrm{eff},b}(\varphi)=\displaystyle{\int^{L}_{-L}}\big(\kappa_{\text{max}}-\kappa(s) \big)\,|\varphi(s)|^{2}\mathrm{~d}s+\hbar^{2}\displaystyle{\int^{L}_{-L}}|\,(\partial_{s}-ib\beta_{0})\varphi |^{2}\mathrm{~d}s\,.$$
Consequently
\begin{equation}
    \label{40}
    q_{h,4}^{\mathrm{eff},b}(\varphi)\geq   \displaystyle{\int^{L}_{-L}}\big(\kappa_{\text{max}}-\kappa(s) \big)\,|\varphi(s)|^{2}\,ds \,.
\end{equation}
From \eqref{33} and \eqref{40}, we have
\begin{align*}
 \displaystyle{\int^{L}_{-L}}\bigg[ \big(\kappa_{\text{max}}-\kappa(s) \big)-\lambda-\hbar^{2}|\phi^{'}|^{2}\bigg]\,e^{2 \phi}\,|\psi|^{2}\,ds
 \leq   0\,.
\end{align*}
From \eqref{41}, we get
\begin{align*}
 \displaystyle{\int^{L}_{-L}}\big[ 1-\epsilon^{2}\big]\big(\kappa_{\text{max}}-\kappa(s) \big)\,e^{2 \phi}\,|\psi|^{2}\mathrm{~d}s
 \leq \lambda \|e^{ \phi}\,\psi\|^{2}  \leq C \hbar \|e^{ \phi}\,\psi\|^{2}\,.
\end{align*}
We notice that for any $ \epsilon \in (0,1)$ and for $ \hbar $ small enough, there exists a constant $ C> 0 $ independent of $ \hbar $ such that
$$1-\epsilon^{2} \geq C \,. $$
We deduce the inequality
\begin{equation}
\label{38}
    \displaystyle{\int^{L}_{-L}}\big(  \kappa_{\text{max}}-\kappa(s)\big)\,e^{2 \phi}\,|\psi|^{2}\mathrm{~d}s
  \leq C \hbar \|e^{ \phi}\,\psi\|^{2}\,.
\end{equation}
From \eqref{38} and \eqref{41}, we have 
\begin{equation}
\label{37}
      q_{h,4}^{\mathrm{eff},b}(e^{ \phi}\,\psi)= \lambda \|e^{ \phi}\,\psi\|^{2}+\hbar^{2}\|\partial_{s}e^{ \phi}\,\psi\|^{2} \leq C \hbar  \|e^{ \phi}\,\psi\|^{2}\,.
\end{equation}
\textbf{3. Third step.} 
 \\ \\
Let us show that:
\begin{equation}
    \label{43}
    \displaystyle{\int_{-L}^{L}} e^{2\epsilon \phi_{0}/\hbar} |\psi|^{2}\mathrm{~d}s \leq C \|\psi\|^{2}\,,
\end{equation}
\begin{equation}
    \label{44}
   q_{h,4}^{\mathrm{eff},b}(e^{\epsilon \phi_{0}/\hbar}\,\psi) \leq C \hbar \|\psi\|^{2}\,.
\end{equation}
We use that $ u=\kappa_{\text{max}}-\kappa$ admits a unique non-degenerate minimum, we deduce that  for any $ C_{0}>0 $, there exist constants $ C, \hbar_{0} $ such that, for all $  \hbar \in (0, \hbar_{0}) \, $,
\begin{equation}
\label{45}
\int_{|s| \leq C_{0} \hbar^{1 / 2}} \big(\kappa_{\text{max}}-\kappa(s)\big) e^{2 \varepsilon \Phi_{0} / \hbar}|\psi|^{2} \mathrm{~d}s\leq C \hbar\|\psi\|^{2}\,,
\end{equation}
\begin{equation}
\label{46}
\int_{|s| \leq C_{0} \hbar^{1 / 2}}  e^{2 \varepsilon \Phi_{0} / \hbar}|\psi|^{2} \mathrm{~d}s\leq C \|\psi\|^{2}\,,
\end{equation}
\begin{equation}
\label{47}
\int_{|s|\geq C_{0} h^{1 / 2}} \big(\kappa_{\text{max}}-\kappa(s)\big)e^{2 \epsilon \phi_{0}/ \hbar}|\psi|^{2} \mathrm{~d}s \geq c \,C_{0}^{2} \hbar \int_{|s| \geq C_{0} \hbar^{1 / 2}} e^{2 \epsilon \phi_{0} / \hbar}|\psi|^{2} \mathrm{~d}s\,.
\end{equation} 
\\
By the inequality \eqref {46}, to show \eqref {43}, it suffices to show the following inequality:
\begin{equation}
    \label{48}
    \int_{|s| \geq C_{0} \hbar^{1 / 2}} e^{2 \epsilon \phi_{0} / \hbar}|\psi|^{2} \mathrm{~d}s \leq  C \|\psi\|^{2}\,.
\end{equation}
Taking $C_{0}$ large enough, we deduce that
\begin{align*}
    \int_{|s| \geq C_{0} \hbar^{1 / 2}} e^{2 \epsilon \phi_{0} / \hbar}|\psi|^{2} \mathrm{~d}s&\leq \dfrac{1}{c \,C_{0}^{2} \hbar} \int_{|s|\geq C_{0} h^{1 / 2}} u(s)e^{2 \epsilon \phi_{0}/ \hbar}|\psi)|^{2} \mathrm{~d}s \\&\leq \dfrac{1}{c \,C_{0}^{2} \hbar}\int_{-L}^{L} u(s)e^{2 \epsilon \phi_{0}/ \hbar}|\psi|^{2} \mathrm{~d}s \,.
    \end{align*}
    By inequality \eqref{38}, we get
    \begin{align*}
        \int_{|s| \geq C_{0} \hbar^{1 / 2}} e^{2 \epsilon \phi_{0} / \hbar}|\psi|^{2} ds &
 \leq \dfrac{C}{c \,C_{0}^{2}}\int_{-L}^{L} e^{2 \epsilon \phi_{0}/ \hbar}|\psi|^{2} \mathrm{~d}s \\ & \leq\dfrac{C}{c \,C_{0}^{2}}\int_{|s| \leq C_{0} \hbar^{1 / 2}} e^{2 \epsilon \phi_{0}/ \hbar}|\psi|^{2} \mathrm{~d}s + \dfrac{C}{c \,C_{0}^{2}}\int_{|s| \geq C_{0} \hbar^{1 / 2}} e^{2 \epsilon \phi_{0}/ \hbar}|\psi|^{2} \mathrm{~d}s \,.
\end{align*}
Thus,
$$ \bigg(1- \dfrac{C}{c \,C_{0}^{2}}\bigg)\int_{|s| \geq C_{0} \hbar^{1 / 2}} e^{2 \epsilon \phi_{0} / \hbar}|\psi|^{2} \mathrm{~d}s  \leq\dfrac{C}{c \,C_{0}^{2}}\int_{|s| \leq C_{0} \hbar^{1 / 2}} e^{2 \epsilon \phi_{0}/ \hbar}|\psi|^{2} \mathrm{~d}s\,.   $$ \\
 From the inequality \eqref{46}, we get \eqref{48} and hence we show the point \eqref {43}.\\
 From \eqref{37}, we have
  $$q_{h,4}^{\mathrm{eff},b}(e^{\epsilon \phi_{0}/\hbar}\,\psi) \leq C \hbar \|e^{\epsilon \phi_{0}/\hbar}\,\psi\|^{2}\,,$$
  and from the inequality \eqref{43}, we get the inequality \eqref {44}.
\end{proof}

We record the following simple corollary of Proposition \ref{123}\,.
\begin{corollary}
  \label{123'}
Let $\rho \in(0,\frac{1}{2})$, there exists  $C>0\,,$  $\epsilon \in (0,1)\,$ and  $\hbar_{0}>0$ such that, for all $h\in (0,\hbar_{0})$ and for all  normalized eigenfunctions $\psi$ of $\mathcal{L}_{h,4}^{\mathrm{eff},b}\,,$ associated with the eigenvalues of orders $ \hbar \,, $ we have:
$$\displaystyle{\int_{|s|\geq \hbar^{\frac{1}{2}-\rho}}} |\psi|^{2}ds \leq C \exp(-c\hbar^{-\epsilon})\,. $$
\end{corollary}
\subsection{Reduction to a flux-free operator}
Consider the Dirichlet effective operator of the form:
$$\mathscr{L}_{h}^{\mathrm{eff}}=-\hbar^{2}\partial_{s}^{2}+\kappa_{\text{max}}-\kappa(s)\,,$$
with domain
$${\rm Dom}\big(\mathscr{L}_{h}^{\mathrm{eff}}\big)= \lbrace \, u\in H^{2}([ -L ,L\,]),\,\, u(-L \, )=u(L\,  )=0\,\rbrace \,,$$
and the quadratic form on $\textit{$L^{2}([-L,L])$}$ associated to it $$q_{h}^{\mathrm{eff}}(\varphi)=\displaystyle{\int^{L}_{-L}}\big(\kappa_{\text{max}}-\kappa(s) \big)\,|\varphi(s)|^{2}\mathrm{~d}s+\hbar^{2}\displaystyle{\int^{L}_{-L}}|\,\partial_{s}\varphi |^{2}\mathrm{~d}s\,.$$
Let $(\lambda_{n}(\mathscr{L}_{h}^{\mathrm{eff}}))_{n \in \mathbb{N}^{*}}\,$ be the sequence of eigenvalues of the operator  $\textit{$\mathscr{L}_{h}^{\mathrm{eff}}$ }$. As a consequence of Corollary \ref{123'}, for small $\hbar$ the eigenfunctions of $\mathcal{L}_{h,4}^{\mathrm{eff}}\,,$ associated with the eigenvalues of orders $ \hbar $ concentrate exponentially near the point of maximum curvature at the scale $\hbar^{1/2}\,.$ The following proposition allows the reduction of the analysis to the operator $\textit{$\mathcal{L}_{h,4}^{\mathrm{eff},b}$}\,.$ It is standard to deduce from the min-max principle and the Agmon estimates of
Proposition \ref{123}, the following proposition. This result quantifies the relation between the spectrum of this operator and that of the original operator.
The notation $\mathcal{O}(h^{\infty})$ indicates a quantity satisfying that,
for all $N\in \mathbb{N}$, there exists $C_{N}>0$ and $h_{N} > 0$ such that, for all $h\in(0,h_{N})$, $|\mathcal{O}(h^{\infty})|\leq C_{N} h_{N}\,.$
 \begin{proposition}
 \label{L2,}
Let $n \in \mathbb{N},$ there exist constant $\,\hbar_{0}\in(0,1)$ such that, for all $\, \hbar \in (0,\hbar_{0})\,,$ $ n \geq 1$ and $\lambda_{n}(\mathcal{L}_{h,4}^{\mathrm{eff},b}) =\mathcal{O}(\hbar)$,
 $$\lambda_{n}(\mathscr{L}_{h}^{\mathrm{eff}})  \leq \lambda_{n}(\mathcal{L}_{h,4}^{\mathrm{eff},b})+\mathcal{O}(h^{\infty})\,.$$
 Moreover, we have, for all $ n \geq 1\,,\,$ $\hbar>0$
 $$\lambda_{n}(\mathcal{L}_{h,4}^{\mathrm{eff},b}) \leq \lambda_{n}(\mathscr{L}_{h}^{\mathrm{eff}}) \,.$$
\end{proposition} 
\begin{proof}
For all  $N \geq 1\,,$ we can take a family of eigenvalues and eigenfunctions
$(\lambda_{n}(\mathcal{L}_{h,4}^{\mathrm{eff},b})_{n=1, \ldots, N}$ such that $ (\psi_{n, \hbar})_{n=1, \ldots, N}$ is an orthonormal family. 
Let us consider the cut off function $\chi$ define on $\mathbb{R}$ as following:
$$0 \leq \chi \leq1\,, \quad  \chi=1 \,\,\, \text{on}\,\, \bigg[-\dfrac{L}{2},\dfrac{L}{2}\bigg]\, \quad \text{and} \quad \chi=0\,\,\, \text{on}\,\,\, \mathbb{R}\setminus ]-L,L[\,.$$
For $k=1,...,N,$ we define the function
$$\Phi_{k,\hbar}(s) =e^{-ib\beta_{0}s}\chi(s)\psi_{j, \hbar}(s)\,.$$
We notice that
$$\Phi_{k,\hbar}(-L)=\Phi_{k,\hbar}(L)=0 \quad \text{and} \quad \Phi_{k,\hbar}\in H^{2}([ -L ,L\,])\,. $$
Let $\mathcal{E}_{N}(\hbar)$ the vector subspace of $\mathrm{Dom}(\mathscr{L}_{h}^{\mathrm{eff}})$ spanned by the family
$ (\Phi_{n, \hbar})_{n=1, \ldots, N}\,,$  and $u_{\hbar} \in F_{h} $ which is written as follows:
  $$u_{\hbar} =\sum_{k=1}^{N} \beta_{k} \Phi_{k, \hbar} \,.$$
  The functions $(\Phi_{n, \hbar}) $ are almost orthonormal. We note that $\dim(\mathcal{E}_{N}(\hbar))=N$ and 
  $$\|u_{\hbar} \|^{2}=\sum_{k=1}^{N} |\beta_{k}|^{2}+\mathcal{O}(h^{\infty} )\,.$$
Inserting $u_{h}$ into the quadratic  form $\textit{$q^{\mathrm{eff}}_{h} $}(u_{h})$ 
 \begin{align*}
 \textit{$q^{\mathrm{eff}}_{h} $}(u_{h})=\langle \textit{$ \mathscr{L}^{\mathrm{eff}}_{h} $ }u_{h},u_{h} \rangle=\sum_{j,k=1}^{N}\beta_{j}\beta_{k} \langle \textit{$ \mathscr{L}^{\mathrm{eff}}_{h} $ }\Phi_{j,\hbar},\Phi_{k,\hbar} \rangle\,.
 \end{align*}
 For $j$ fixed and for $ h $ small enough, we have
 \begin{align*}
  &\textit{$ \mathscr{L}^{\mathrm{eff}}_{h} $ }\Phi_{j,\hbar}\\&=
  e^{-ib\beta_{0}s}\big[ -\hbar^{2}(\partial_{s}-ib\beta_{0})^{2}\chi\psi_{j, \hbar} +(\kappa_{\text{max}}-\kappa(s))\chi\psi_{j, \hbar}\big]\\&=  e^{-ib\beta_{0}s}\big[ -\hbar^{2}\chi(s)(\partial_{s}-ib\beta_{0})^{2}\psi_{j, \hbar}+(\kappa_{\text{max}}-\kappa(s))\chi\psi_{j, \hbar}-2\hbar^{2}\partial_{s}\chi\partial_{s}\psi_{j, \hbar}-\hbar^{2}\psi_{j, \hbar}\partial_{s}^{2}\chi\\&\quad\quad\quad\quad\quad\quad\quad\quad\quad\quad\quad\quad\quad\quad\quad\quad\quad\quad\quad\quad\quad\quad\quad\quad\quad\quad\quad\quad\quad\quad\quad+2\hbar^{2}ib\beta_{0}\psi_{j, \hbar}\partial_{s}\chi \big] \\&=  e^{-ib\beta_{0}s}\big[ -\hbar^{2}\chi(\partial_{s}-ib\beta_{0})^{2}\psi_{j, \hbar}+(\kappa_{\text{max}}-\kappa(s))\chi\psi_{j, \hbar}-\hbar^{2}\big( 2 \partial_{s}\chi [\partial_{s}- ib\beta_{0}]\psi_{j, \hbar}+\psi_{j, \hbar}\partial_{s}^{2}\chi\big)\big]\\&= e^{-ib\beta_{0}s}\big[ \chi\mathcal{L}_{h,4}^{\mathrm{eff,b}}\psi_{j, \hbar} -\hbar^{2}\big( 2 \partial_{s}\chi [\partial_{s}- ib\beta_{0}]\psi_{j, \hbar}+\psi_{j, \hbar}\partial_{s}^{2}\chi\big)\big]\,.
 \end{align*}
 Consequently, for $j,k$ fixed and for $ h $ small enough, we have
 \begin{align*}
     \langle \textit{$ \mathscr{L}^{\mathrm{eff}}_{h} $ }\Phi_{j,\hbar} , \Phi_{k,\hbar}  \rangle&= \langle   \chi\mathcal{L}_{h,4}^{\mathrm{eff,b}}\psi_{j, \hbar},\chi\psi_{k, \hbar}\rangle-\hbar^{2}\langle \big( 2 \partial_{s}\chi [\partial_{s}- ib\beta_{0}]\psi_{j, \hbar}+\psi_{j, \hbar}\partial_{s}^{2}\chi\big),\chi\psi_{k, \hbar}\rangle\\&=\delta_{j,k} \lambda_{j}(\mathcal{L}_{h,4}^{\mathrm{eff,b}})-\hbar^{2}\langle \big( 2 \partial_{s}\chi [\partial_{s}- ib\beta_{0}]\psi_{j, \hbar}+\psi_{j, \hbar}\partial_{s}^{2}\chi\big),\chi\psi_{k, \hbar}\rangle\,.
 \end{align*}
 From Hölder's inequality and the Proposition \ref{123}, we obtain
 $$\hbar^{2}\big|\langle \big( 2 \partial_{s}\chi [\partial_{s}- ib\beta_{0}]\psi_{j, \hbar}+\psi_{j, \hbar}\partial_{s}^{2}\chi\big),\chi\psi_{k, \hbar}\rangle\big|=\mathcal{O}(\hbar^{\infty})\,.$$
  From the previous inequality, we get
  $$\langle \textit{$ \mathscr{L}^{\mathrm{eff}}_{h} $ }\Phi_{j,\hbar} , \Phi_{k,\hbar}  \rangle=\delta_{j,k} \lambda_{j}(\mathcal{L}_{h,4}^{\mathrm{eff},b})+\mathcal{O}(\hbar^{\infty})\,.$$
  Therefore,
 \begin{align*}
  \textit{$q^{\mathrm{eff}}_{h} $}(u_{h})&=\sum_{j=1}^{N}\lambda_{j}(\mathcal{L}_{h,4}^{\mathrm{eff},b})|\beta_{j}|^{2}+ \mathcal{O}(\hbar^{\infty})
  \sum_{j,k=1}^{N}\beta_{j}\beta_{k} \leq \bigg(\lambda_{N}(\mathcal{L}_{h,4}^{\mathrm{eff},b})+ o(\hbar)\,\bigg) \|u_{h}\|^{2}\,.
 \end{align*}
By the min-max theorem, we have
 $\,\,\,\,\lambda_{N}(\mathscr{L}_{h}^{\mathrm{eff}})\leq \lambda_{N}(\mathcal{L}_{h,4}^{\mathrm{eff},b})+\mathcal{O}(\hbar^{\infty})\,.$\\
 
  Now we determine an upper bound of $\lambda_{n}(\mathscr{L}_{h}^{\mathrm{eff}})$ in terms of $\lambda_{n}(\mathcal{L}_{h,4}^{\mathrm{eff},b})\,.$
  For all  $N \geq 1\,,$ we can take a family of eigenvalues and eigenfunctions
$(\lambda_{n}(\mathscr{L}_{h}^{\mathrm{eff}}))_{n=1, \ldots, N}$ such that $ (u_{n, \hbar})_{n=1, \ldots, N}$ is an orthonormal family. 
For $k=1,...,N,$ we define the function
$$\Psi_{k,\hbar}(s) =e^{ib\beta_{0}s}u_{k, \hbar}(s)\,.$$
We notice that
$$\Psi_{k,\hbar}(-L)=\Psi_{k,\hbar}(L)=0 \quad \text{and} \quad \Psi_{k,\hbar}\in H^{2}(] -L ,L\,[)\,. $$
Let $\mathcal{E}_{N}(\hbar)$ the vector subspace of $\mathrm{Dom}(\mathcal{L}_{h,4}^{\mathrm{eff},b})$ spanned by the family $ (\Psi_{n, \hbar})_{n=1, \ldots, N}\,.$  
  Let $\Psi \in \mathcal{E}_{N}(\hbar)$ which is written as follows:
  $$ \Psi =\sum_{k=1}^{N} \beta_{k} \Psi_{k, \hbar} =\sum_{k=1}^{N} \beta_{k}  e^{ib\beta_{0}s}u_{k, \hbar}(s)=e^{ib\beta_{0}s} u\,, \quad  u=\sum_{k=1}^{N} \beta_{k} u_{k, \hbar}(s)\,.$$
  We have 
  $$\mathcal{L}_{h,4}^{\mathrm{eff},b}\Psi=e^{ib\beta_{0}s}\,\mathscr{L}_{h}^{\mathrm{eff}} u\,,$$
  then 
  $$\langle \mathcal{L}_{h,4}^{\mathrm{eff},b}\Psi,\Psi  \rangle =\langle \mathscr{L}_{h}^{\mathrm{eff}} u,u \rangle \leq \lambda_{N}(\mathscr{L}_{h}^{\mathrm{eff}}) \|u\|^{2}= \lambda_{N}(\mathscr{L}_{h}^{\mathrm{eff}}) \|\Psi\|^{2}\,. $$
  By the min-max theorem, we get the non-asymptotic inequality
   $$\lambda_{N}(\mathcal{L}_{h,4}^{\mathrm{eff},b}) \leq \lambda_{N}(\mathscr{L}_{h}^{\mathrm{eff}}) \,.$$
\end{proof}
\subsection{End of the proof}
In this section we will construct trial states $(\Phi_{n})$ in the domain of the effective operator $\mathscr{L}_{h}^{\mathrm{eff}}\,,$ 
such that, for every fixed $n \in \mathbb{N}^{*}$  we have\\
\begin{equation}
\label{49}
\bigg\|\mathscr{L}_{h}^{\mathrm{eff}} \Phi_{n}-(2n-1) \hbar \sqrt{\frac{-\kappa^{\prime \prime}(0)}{2}}  \Phi_{n}\bigg\|_{\textit{$L^{2}(\mathbb{R}/2L\mathbb{Z})$}}\leq C \hbar^{3/2} \,.
\,.
\end{equation}
We know that, the eigenfunctions are localized near the boundary at the scale $\hbar^{\frac{1}{2}}$ (cf. Prop \ref{123};the same proof can work for $\mathscr{L}_{h}^{\mathrm{eff}} $) . It suggests to introduce the rescaling
\begin{equation}
    \label{111}
   s= \hbar^{1/2} \sigma\,. 
\end{equation}
The effective operator $\mathscr{L}_{h}^{\mathrm{eff}} \,$ becomes
$$\mathscr{L}_{h}^{\mathrm{eff}} =-\hbar\, \partial_{\sigma}^{2}+\kappa_{\max}-\kappa(\hbar^{\frac{1}{2}} \sigma)\,.$$
Recall that the value of the maximum curvature is $\kappa_{\max}=\kappa(0)$. We can use a  Taylor expansion of $\kappa$ near $0\,,$ as follows
$$\kappa(\hbar^{\frac{1}{2}} \sigma)=\kappa_{\max}+\dfrac{\hbar \sigma^{2}}{2}\kappa^{\prime \prime}(0)+\hbar^{\frac{3}{2}} q(\sigma)\,,
$$
where the functions $q$ satisfy for $|\sigma|=\mathcal{O}(\hbar^{-\frac{1}{2}})$
\begin{equation}
    \label{q}
   |q(\sigma)\leq C|\sigma|^{3}\,. 
\end{equation}
Consequently, we write the effective operator $\mathscr{L}_{h}^{\mathrm{eff}}\,$ as
\begin{equation}
    \label{50}
  \hbar^{-1}\mathscr{L}_{h}^{\mathrm{eff}}=- \partial_{\sigma}^{2}-\dfrac{ \sigma^{2}}{2}\kappa^{\prime \prime}(0)-\hbar^{\frac{1}{2}} q(\sigma)\,.
\end{equation}
Since  $\kappa(0)$ is the non-degenerate maximum of $\kappa$, we have $\kappa^{\prime \prime}(0)<0$. The eigenvalues of the harmonic oscillator $H_{\text{harm}}$
\begin{equation}
    \label{harmonique}
    -\partial_{\sigma}^{2}+\dfrac{-\kappa^{\prime \prime}(0)}{2} \sigma^{2} \quad \text { in }  L^{2}(\mathbb{R})
\end{equation}
are  $(2 n-1) \sqrt{\dfrac{-\kappa^{\prime \prime}(0)}{2}}$ with $n \in \mathbb{N}^{*}\,.$
The corresponding normalized eigenfunctions are denoted by $f_{n}(\sigma)$. They have the form
\begin{equation}
    \label{51}
    f_{n}(\sigma)=h_{n}(\sigma) \exp \left(-\sqrt{-\frac{k^{\prime \prime}(0)}{2}} \frac{\sigma^{2}}{2}\right)\,,
\end{equation}
where the  $(h_{n}(\sigma))_{n\geq1}$ are the rescaled Hermite polynomials.\\

Define the trial state  $\Phi_{n}$ in  $\textit{$L^{2}(]-L,L[)$}$ as follows:
\begin{equation}
    \label{52}
   \Phi_{n}(s)=\hbar^{-1/4}\,\chi\left(\dfrac{s}{2L} \right) f_{n}(\hbar^{-1/2}s)=\hbar^{-1/4}\,u_{n}(s)\,, 
\end{equation}
where $\chi$ is a cut-off function defined on $\mathbb{R}$ by:
 $$0\leq \chi\leq1 \,\,,\,\, \chi=1 \,\,\text{ in } \,\,]-1/4, 1/4\,]\,\, \text{ and } \,\,\chi=0\,\, \text{ in } \,\,\mathbb{R}\backslash]-1/2,1/2[,$$
 and $$u_{n}(s)= \chi\left(\dfrac{s}{2L} \right) f_{n}(\hbar^{-1/2}s)\,.$$
Using the exponential decay of $f_{n}$ at infinity, it is
easy to verify
\begin{equation}
    \label{53}
    \left\|\Phi_{n}\right\|_{\textit{$L^{2}(]-L,L[)$}}=1+\mathcal{O}(\hbar^{\infty})\,.
\end{equation}
Then, we observe that after a change of variables and by construction of $\Phi_{n},$ \eqref{50} yields
\begin{equation}\label{54}
\begin{aligned}
    &\hbar^{-1}\mathscr{L}_{h}^{\mathrm{eff}}\Phi_{n}(s)\\&=-\hbar^{-1/4}\Bigg[\partial_{\sigma}^{2}u_{n}(\sigma)
 -\dfrac{ \sigma^{2}}{2}\kappa^{\prime \prime}(0)u_{n}(\sigma)-\hbar^{\frac{1}{2}} q(\sigma)u_{n}(\sigma)\Bigg]\\&=-\hbar^{-1/4}\Bigg[ \chi\left(\dfrac{\hbar^{1/2}\sigma}{2L} \right) \partial_{\sigma}^{2}f_{n}(\sigma)+2\hbar^{\frac{1}{2}} \chi^{\prime}\left(\dfrac{\hbar^{\frac{1}{2}}\sigma}{2L} \right) f^{\prime}_{n}(\sigma)+\hbar \chi^{\prime\prime}\left(\dfrac{\hbar^{\frac{1}{2}}\sigma}{2L} \right) f_{n}(\sigma)\\&\quad 
 -\dfrac{ \sigma^{2}}{2}\kappa^{\prime \prime}(0)u_{n}(\sigma)-\hbar^{\frac{1}{2}} q(\sigma)u_{n}(\sigma)\Bigg]\\&= -\hbar^{-1/4}\Bigg[ -(2n-1) \sqrt{\frac{-\kappa^{\prime \prime}(0)}{2}}u_{n}(\sigma)+
2\hbar^{\frac{1}{2}} \chi^{\prime}\left(\dfrac{\hbar^{\frac{1}{2}}\sigma}{2L} \right) f^{\prime}_{n}(\sigma)+\hbar \chi^{\prime\prime}\left(\dfrac{\hbar^{\frac{1}{2}}\sigma}{2L} \right) f_{n}(\sigma) \\&\quad
 -\hbar^{\frac{1}{2}} q(\sigma)u_{n}(\sigma)
    \Bigg]\\&=(2n-1) \sqrt{\frac{-\kappa^{\prime \prime}(0)}{2}}  \Phi_{n}(s) -\hbar^{-1/4}\Bigg[ 
\bigg(2\hbar^{\frac{1}{2}} \chi^{\prime}\left(\dfrac{\hbar^{\frac{1}{2}}\sigma}{2L} \right) f^{\prime}_{n}(\sigma)+\hbar \chi^{\prime\prime}\left(\dfrac{\hbar^{\frac{1}{2}}\sigma}{2L} \right) f_{n}(\sigma) \bigg)
 \\&\quad-\hbar^{\frac{1}{2}} q(\sigma)u_{n}(\sigma)
    \Bigg]\,.
\end{aligned}
\end{equation}
Using  the exponential decay of  $f_{n}$ at infinity, we infer from \eqref{54} that
$$\bigg\|\hbar^{-1}\mathscr{L}_{h}^{\mathrm{eff}} \Phi_{n}-(2n-1)  \sqrt{\frac{-\kappa^{\prime \prime}(0)}{2}}  \Phi_{n}\bigg\|_{\textit{$L^{2}(]-L,L[)$}}\leq C \hbar^{1/2} \,.$$

 As a consequence of the spectral theorem, \eqref{49} yields that, for
every fixed $n \in \mathbb{N}^{*}$,  there exists an eigenvalue $\Tilde{\lambda}_{n}(\mathscr{L}_{h}^{\mathrm{eff}})$ of the operator $\mathscr{L}_{h}^{\mathrm{eff}}$ such that
\begin{equation}
    \label{55}
  \Tilde{\lambda}_{n}(\mathscr{L}_{h}^{\mathrm{eff}})=  (2n-1) \hbar \sqrt{\frac{-\kappa^{\prime \prime}(0)}{2}} +\mathcal{O}(\hbar^{3/2})\,.
\end{equation}
Now, we proved by finding lower bounds.
\begin{proposition}
  \label{a²}
  For all $n \geq 1\,$ there exist $h_{0}>0$ such that, for all  $h\in (0,h_{0})$
   $$\lambda_{n}(\mathscr{L}_{h}^{\mathrm{eff}}) \geq (2n-1) \hbar \sqrt{\frac{-\kappa^{\prime \prime}(0)}{2}} +\mathcal{O}(\hbar^{3/2})\,.$$
\end{proposition}
\begin{proof}
For all  $N \geq 1\,,$ we can take a family of eigenvalues and eigenfunctions
\\$(\lambda_{n}(\mathscr{L}_{h}^{\mathrm{eff}}), \psi_{n, \hbar})_{n=1, \ldots, N}$ such that $ (\psi_{n, \hbar})_{n=1, \ldots, N}$ is an orthonormal family. Let $\mathcal{E}_{N}(\hbar)$ the vector subspace of $\mathrm{Dom}(\mathscr{L}_{h}^{\mathrm{eff}})$ of dimension $N$ spanned by the family
$ (\psi_{n, \hbar})_{n=1, \ldots, N}\,.$ \\
It is rather easy to observe that, 
for $\psi \in \mathcal{E}_{N}(\hbar)$
$$ \langle \mathscr{L}_{h}^{\mathrm{eff}} \psi,\psi\rangle \leq \lambda_{N}(\mathscr{L}_{h}^{\mathrm{eff}} ) \|\psi\|^{2}_{\textit{$L^{2}(]-L,L[)$}}\,.  $$  Since $\psi(-L)=\psi(L)=0, $ the following function is in $H^{1}(\mathbb{R})$
  $$\tilde{\psi}= \begin{cases} \psi & \text{if $s \in ]-L,L[$ } \\ 0 & \text{if $s \in \mathbb{R}\setminus ]-L,L[ $} \,. \end{cases}$$ 
  For $N \geq1,$ let 
  $$\tilde{\mathcal{E}}_{N}(\hbar)=\text{Span}\lbrace \tilde{\psi}_{n, \hbar}\rbrace_{n=1, \ldots, N} \,. $$
  Since $\text{dim}(\mathcal{E}_{N}(\hbar))=N,$ then $\text{dim}(\tilde{\mathcal{E}}_{N}(\hbar))=N\,.$ We notice that, for all $\psi \in \mathcal{E}_{N}(\hbar)$, $\psi$ also satifies  the Agmon estimates (Proposition  \ref{123}). Using the change of variable introduced in \eqref{111}, we easily obtain that, there exists $ C>0 $ such that,   for all $ \psi \in \mathcal{E}_{N} (\hbar) \, $ 
  $$\displaystyle{\int_{-\hbar^{-\frac{1}{2}}L}^{\hbar^{-\frac{1}{2}}L}}|\sigma|^{3}|\psi(\sigma)|^{2}\mathrm{~d}(\hbar^{1/2}\,\sigma)\leq C \hbar^{\frac{3}{2}} \|\psi\|^{2}_{\textit{$L^{2}(]-L,L[)$}}\,.$$
  Consequently, we get
\begin{align*}
   &\lambda_{N}(\mathscr{L}_{h}^{\mathrm{eff}} ) \|\psi\|^{2}_{\textit{$L^{2}(]-L,L[)$}}\\ &\geq \langle \mathscr{L}_{h}^{\mathrm{eff}} \psi,\psi\rangle\\ &\geq \displaystyle{\int_{\mathbb{R}}}\hbar\, \dfrac{-\kappa^{\prime \prime}(0) }{2}\sigma^{2}\,|\tilde{\psi}( \hbar^{1/2} \sigma)|^{2}\mathrm{~d}(\hbar^{1/2}\,\sigma)+\hbar\displaystyle{\int_{\mathbb{R}}}|\, \partial_{\sigma}\tilde{\psi} |^{2}\mathrm{~d}(\hbar^{1/2}\,\sigma) - C\hbar^{3/2}\|\psi\|^{2}_{\textit{$L^{2}(]-L,L[)$}} \\ & \geq \hbar \bigg[ \displaystyle{\int_{\mathbb{R}}} \bigg(|\,\partial_{\sigma}\tilde{\psi}( \hbar^{1/2} \sigma) |^{2} +
\, \dfrac{-\kappa^{\prime \prime}(0) }{2}\sigma^{2}\,|\tilde{\psi}( \hbar^{1/2} \sigma)|^{2}\,  \bigg)\mathrm{~d}(\hbar^{1/2}\,\sigma)\bigg]- C\hbar^{3/2}\|\psi\|^{2}_{\textit{$L^{2}(]-L,L[)$}}\,.
\end{align*}
This becomes
\begin{align*}
    \hbar \bigg[ \displaystyle{\int_{\mathbb{R}}} \bigg(|\,\partial_{\sigma}\tilde{\psi}( \hbar^{1/2} \sigma) |^{2} +
\, \dfrac{-\kappa^{\prime \prime}(0) }{2}\sigma^{2}\,|\tilde{\psi}( \hbar^{1/2} \sigma)|^{2}\,  \bigg)\mathrm{~d}(\hbar^{1/2}\,\sigma)\bigg]&\leq \lambda_{N}(\mathscr{L}_{h}^{\mathrm{eff}} ) \|\psi\|^{2}_{\textit{$L^{2}(]-L,L[)$}}\\& \quad\quad\quad\quad+\mathcal{O}(\hbar^{3/2})\|\psi\|^{2}_{\textit{$L^{2}(]-L,L[)$}}  \,.
\end{align*}
We deduce that $$\max_{\tilde{\psi}  \in \tilde{\mathcal{E}}_{N}(\hbar) }\dfrac{\langle H_{\text{harm}} \tilde{\psi},\tilde{\psi}  \rangle_{\textit{$L^{2}(\mathbb{R})$}}}{\|\tilde{\psi}\|^{2}_{\textit{$L^{2}(\mathbb{R})$}}}\leq \lambda_{N}(\mathscr{L}_{h}^{\mathrm{eff}} ) +\mathcal{O}(\hbar^{3/2})\,.  $$
Thus, by the min-max principle, we have
the comparison of the Rayleigh quotients
\begin{align*}
  \lambda_{N}(\mathscr{L}_{h}^{\mathrm{eff}} )& \geq  \lambda_{N}(H_{\text{harm}} )   +\mathcal{O}(\hbar^{3/2})
\\ &\geq  (2N-1) \hbar \sqrt{\frac{-\kappa^{\prime \prime}(0)}{2}}+\mathcal{O}(\hbar^{3/2})\,.
\end{align*}
\end{proof}
Finally, from \eqref{55}, the asymptotic expansion of the $n$-th eigenvalue of $\mathscr{L}_{h}^{\mathrm{eff}} $ is:
  $$\lambda_{n}(\mathscr{L}_{h}^{\mathrm{eff}} )=(2n-1) \hbar \sqrt{\frac{-\kappa^{\prime \prime}(0)}{2}} +\mathcal{O}(\hbar^{3/2})\,.$$ 
Now, by inequalities \eqref{jiu} and Proposition \ref{L2,}, we deduce the asymptotic expansion of the $n$-th eigenvalue of the effective operator
$\mathcal{L}_{h,\alpha,0}^{\mathrm{eff}}$ 
\begin{equation}
    \label{56}
    \lambda_{n}(\mathcal{L}_{h,\alpha,0}^{\mathrm{eff}})=-1-\kappa_{\text{max}}\hbar^{2}+(2n-1)  \sqrt{\frac{-\kappa^{\prime \prime}(0)}{2}} \hbar^{3}+\mathcal{O}(\hbar^{7/2})\,.
\end{equation}
From Theorem \ref{R21} and Proposition \ref{L2}, we have shown Corollary \ref{ab},  which was a reformulation of Corollary \ref{abc}.
\section{Unit disc case}
\label{5é}
In this section we analyse the magnetic Robin Laplacian with a negative boundary parameter on the disc. There by proving of Theorem \ref{86}. We will study the spectrum of the effective operator $\mathcal{L}_{h,\text{disc}}^{\mathrm{eff},\pm}$, where $\Omega=D(0,1)\,,$ $0< c_{1}<c_{2}\,,$ $0<\alpha<1-\eta\,$ and $ \,\eta<1.$
We have 
$$ L=\dfrac{|\partial\Omega|}{2}=\pi, \quad \beta_{0}=\dfrac{|\Omega|}{|\partial\Omega|}=\dfrac{1}{2} \,, \quad c_{1}h^{\frac{-\eta}{2}}\leq b \leq c_{2}h^{\frac{-\eta}{2}}  \quad  \text {and}  \quad \kappa=1
\,.$$
Thus, the effective operators can be written as follows
$$\mathcal{L}_{h,\text{disc}}^{\mathrm{eff},\pm}=-h(1\pm c_{\pm}\,h^{\min(\alpha,\frac{1}{2})})\left(\partial_{s}-i\dfrac{b}{2}\right)^{2}-1- h^{\frac{1}{2}}-\frac{h}{2}\,,$$
  in  $\textit{$L^{2}(\mathbb{R}/2\pi\mathbb{Z})$}\,,$ where $c_{\pm}$ are constants independent of $h$. The associate quadratic form  on $\textit{$H^{1}(\mathbb{R}/2\pi \mathbb{Z})$}$  are
$$
q_{h,\text{disc}}^{\mathrm{eff},\pm}(\varphi)=\left(-1-h^{\frac{1}{2}}-\frac{h}{2}\right) \int_{-\pi}^{\pi}|\varphi(s)|^{2}\mathrm{~d} s+h\left((1\pm c_{\pm}h^{\min(\alpha,\frac{1}{2})}\right) \int_{-\pi}^{\pi}\left|\left(\partial_{s}-i \frac{b}{2}\right) \varphi\right|^{2}\mathrm{~d}s\,.
$$
We use the completeness of the orthonormal family  $\left\{e^{i m s} / \sqrt{2 \pi}\right\}_{m \in \mathbb{Z}}$ in $L^{2}(\mathbb{R} / 2 \pi \mathbb{Z})\,.$
For all  $\varphi \in H^{1}(\mathbb{R} / 2 \pi \mathbb{Z})$, we have
$$
\varphi(s)=\sum_{m \in  \mathbb{Z}} \widehat{\varphi}(m) \frac{e^{i m s}}{\sqrt{2 \pi}}\,,
$$
where 
$$
\widehat{\varphi}(m)=\frac{1}{\sqrt{2 \pi}} \int_{-\pi}^{\pi} \varphi(s) e^{-i m s} \mathrm{~d} s .
$$
\subsection{Reduction to a simple operator:}
Consider the effective operator 
$$\mathcal{L}_{h,\text{disc}}^{\mathrm{eff}}=-h\left(\partial_{s}-i\dfrac{b}{2}\right)^{2}-1- h^{\frac{1}{2}}-\frac{h}{2}\,,$$
and the associate quadratic form  on $\textit{$H^{1}(\mathbb{R}/2\pi \mathbb{Z})$}$  
$$
q_{h,\text{disc}}^{\mathrm{eff}}(\varphi)=\left(-1-h^{\frac{1}{2}}-\frac{h}{2}\right) \int_{-\pi}^{\pi}|\varphi(s)|^{2}\mathrm{~d} s+h \int_{-\pi}^{\pi}\left|\left(\partial_{s}-i \frac{b}{2}\right) \varphi\right|^{2}\mathrm{~d}s\,.
$$
\begin{proposition}
\label{yiyi}
There exist constants $C > 0$ and $h_{0}\in(0, 1)$ such that, for all $h \in (0, h_{0})$ and $n\in\mathbb{N}^{*}$, it holds
$$|\lambda_{1}(\mathcal{L}_{h,\rm disc}^{\mathrm{eff},\pm})-\lambda_{1}(\mathcal{L}_{h,\rm disc}^{\mathrm{eff}}) |\leq C h^{\min(1+\alpha,\frac{3}{2})}\,.$$
\end{proposition}
\begin{proof}
For all $\varphi \in \textit{$H^{1}(\mathbb{R}/2\pi \mathbb{Z})$}\,,$ we have 
\begin{align*}
    |q_{h,\text{disc}}^{\mathrm{eff}}(\varphi)-q_{h,\text{disc}}^{\mathrm{eff},\pm}(\varphi)|&=c_{\pm}h^{\min(1+\alpha,\frac{3}{2})} \int_{-\pi}^{\pi}\left|\left(\partial_{s}-i \frac{b}{2}\right) \varphi\right|^{2}\mathrm{~d}s\\&=c_{\pm}h^{\min(1+\alpha,\frac{3}{2})}\sum_{m \in  \mathbb{Z}}\left(m-\frac{b}{2}\right)^{2}|\hat{\varphi}(m)|^{2}\\& \leq c_{\pm}h^{\min(1+\alpha,\frac{3}{2})} \inf _{m \in \mathbb{Z}}\left(m-\frac{b}{2}\right)^{2} \|\varphi\|^{2}_{\textit{$L^{2}(\mathbb{R}/2\pi\mathbb{Z})$}}\,.
\end{align*}
We have $$  \inf _{m \in \mathbb{Z}}\left(m-\frac{b}{2}\right)^{2}\leq   \left( \bigg\lfloor \frac{b}{2} \bigg\rfloor+1-\frac{b}{2}\right)^{2}\leq1\,.$$
The conclusion of the proposition is now a simple application of the min-max principle.
\end{proof}
\subsection{Spectrum of $\mathcal{L}_{h,\rm disc}^{\mathrm{eff}}\,$: } For all $\varphi \in H^{1}(\mathbb{R} / 2 \pi \mathbb{Z}),$ we have
$$
\begin{aligned}
q_{h,\text{disc}}^{\mathrm{eff}}(\varphi) &=\left(-1-h^{\frac{1}{2}}-\frac{h}{2}\right) \int_{-\pi}^{\pi}|\varphi(s)|^{2} \mathrm{~d}s+h \int_{-\pi}^{\pi}\left|\left(\partial_{s}-i \frac{b}{2}\right) \varphi\right|^{2} \mathrm{~d} s \\
&=\sum_{m \in  \mathbb{Z}}\left(-1-h^{\frac{1}{2}}-\frac{h}{2}+h\left(m-\frac{b}{2}\right)^{2}\right)|\hat{\varphi}(m)|^{2}
\end{aligned}
$$
Let $k \in \mathbb{Z};$ we set
$$
\varphi_{k}(s)=\frac{e^{i k s}}{\sqrt{2 \pi}}:=e_{k}(s)\,.
$$
Inserting $\varphi_{k}$ into the quadratic form, we have for all  $k \in \mathbb{Z}$
$$
q_{h,\text{disc}}^{\mathrm{eff}}(\varphi_{k})=-1-h^{\frac{1}{2}}-\frac{h}{2}+h\left(k-\frac{b}{2}\right)^{2}.
$$
By application of min-max principle, we have 
\begin{align*}
\forall k \in \mathbb{Z}, \quad \lambda_{1}\left(\mathcal{L}_{h,\text{disc}}^{\mathrm{eff}}\right) & \leq \frac{q_{h,\text{disc}}^{\mathrm{eff}}\left(e_{k}\right)}{\left\|e_{k}\right\|^{2}} \leq-1-h^{\frac{1}{2}}-\frac{h}{2}+h\left(k-\frac{b}{2}\right)^{2},
\end{align*}
then
\begin{align*}
\lambda_{1}(\mathcal{L}_{h,\text{disc}}^{\mathrm{eff}}) & \leq-1-h^{\frac{1}{2}}-\frac{h}{2}+h \inf _{k \in \mathbb{Z}}\left(k-\frac{b}{2}\right)^{2}.
\end{align*}
Therefore, we have the following upper bound on the first eigenvalue
\begin{equation}
    \label{57}
   \lambda_{1}(\mathcal{L}_{h,\text{disc}}^{\mathrm{eff}})  \leq-1-h^{\frac{1}{2}}+\left(\inf _{m \in \mathbb{Z}}\left(m-\frac{b}{2}\right)^{2}-\frac{1}{2}\right) h\,. 
\end{equation}Now, for lower bound. For all  $\varphi \in H^{1}(\mathbb{R} / 2 \pi \mathbb{Z}),\, $ we have
\begin{align*}
    q_{h,\text{disc}}^{\mathrm{eff}}(\varphi) &=\sum_{m \in  \mathbb{Z}}\left(-1-h^{\frac{1}{2}}-\frac{h}{2}+h\left(m-\frac{b}{2}\right)^{2}\right)|\hat{\varphi}(m)|^{2}\\& \geq \left(-1-h^{\frac{1}{2}}-\frac{h}{2}+h\inf _{m \in \mathbb{Z}}\left(m-\frac{b}{2}\right)^{2}\right)\sum_{m \in  \mathbb{Z}}|\hat{\varphi}(m)|^{2}
    \\ & = \left(-1-h^{\frac{1}{2}}-\frac{h}{2}+h\inf _{m \in \mathbb{Z}}\left(m-\frac{b}{2}\right)^{2}\right) \|\varphi\|^{2}_{\textit{$L^{2}(\mathbb{R}/2\pi\mathbb{Z})$}}\,.
\end{align*}
Then, by the min-max principle, we have
\begin{equation}
    \label{58}
    \lambda_{1}(\mathcal{L}_{h,\text{disc}}^{\mathrm{eff}})  \geq-1-h^{\frac{1}{2}}+\left(\inf _{m \in \mathbb{Z}}\left(m-\frac{b}{2}\right)^{2}-\frac{1}{2}\right) h\,. 
\end{equation}
From \eqref{57},\eqref{58} and Proposition \ref{yiyi}, we deduce the asymptotic expansion of the first eigenvalue of the effective operator $\mathcal{L}_{h,\text{disc}}^{\mathrm{eff},\pm}$ 
\begin{equation}
    \label{59}
    \lambda_{1}(\mathcal{L}_{h,\text{disc}}^{\mathrm{eff},\pm}) = -1-h^{\frac{1}{2}}+\left(\inf _{m \in \mathbb{Z}}\left(m-\frac{b}{2}\right)^{2}-\frac{1}{2}\right) h+o(h)\,. 
\end{equation}
 According to Theorem \ref{R21}, and for $\eta<1\,,$ we get 
 $$\tilde{\mathcal{O}}(h^{2-\eta})=o(h)\,,$$ and we deduce
 $$\mu_{1}(h,b)=-h-h^{\frac{3}{2}}+\left(\inf _{m \in \mathbb{Z}}\left(m-\frac{b}{2}\right)^{2}-\frac{1}{2}\right) h^{2}+o(h^{2})\,.$$
Note that, in the case with constant magnetic field, we recover by a very simple proof, Theorem 1.1 of the article \cite{2}.
\begin{appendix}
\section{Auxiliary operators}
\label{oi1}
The aim of this section is to recall some spectral properties related to the Robin Laplacian in dimension one. This model naturally arises in our strategy of dimensional reduction and already appeares in  \cite{6}, \cite{7}, \cite{2}, \cite{14}, \cite{17} and \cite{8}.\subsection{1D Laplacian on the half line.}\label{s6}
We recall here some spectral properties of Robin Laplacian on $\mathbb{R}_{+}$.
Let the operator $$\mathcal{H}_{0,0}:=-\dfrac{d^{2}}{d\tau^{2}}\,\,\, \text{dans}\,\,\, \textit{$L^{2}(\mathbb{R}_{+})$}\, ,$$ with domain $\mathrm{Dom}(\mathcal{H}_{0,0}) =  \lbrace \,u \in  \textit{$H^{2}(\mathbb{R}_{+})$} :\,\,u^{'}(0)\,\,=\,\, - \,\,u(0)\, \rbrace\, . $ \\

This operator is self-adjoint but is not with a compact resolvent, then
the spectrum of this operator is
$\sigma(\mathcal{H}_{0})=\text{sp}_{\text{dis}}(\mathcal{H}_{0,0}) \cup \text{sp}_{\text{ess}}(\mathcal{H}_{0,0})=\lbrace -1\rbrace \cup\left[0,+\infty \right[ \,.\,$  The subspace associated with the unique eigenvalue $ -1 $ is generated by the normalized eigenfunction in $\textit{L}^{2}(\mathbb{R}^{+})\,,$
\begin{equation}
    \label{66}
    u_{0}(\tau)=\sqrt{2}\, \exp(- \, \tau)\, .
    \end{equation}
    \subsection{ A weighted 1D Laplacian }\label{d1}
   Let $B \in \mathbb{R}$, $T>0$ such that $|B|T<\frac{1}{3}\,$. In the weighted space $\textit{$L^{2}(\,(0,T),(1-B\tau)\mathrm{~d}\tau\,)$} $, we introduce the operator
  $$\mathcal{H}_{B}^{\lbrace T \rbrace}=-\dfrac{d^{2}}{d\tau^{2}}+ \dfrac{B}{1-B\tau}\dfrac{d}{d\tau}\,\, ,$$ 
   with domain $$\mathrm{Dom}(\mathcal{H}_{B}^{\lbrace T \rbrace})= \lbrace\, u \in  \textit{$H^{2}(0,T )$} :\,\,u^{'}(0)\,\,=\,-\,u(0)\, \,\text{and}\,\, u(T)\,\,=0 \, \rbrace\,. $$
 This weight will come to measure the effect of the curvature. The operator
 $\mathcal{H}_{B}^{\lbrace T \rbrace}$ is the self-adjoint operator in $\textit{$L^{2}(\,(0,T),(1-B\tau)\,d\tau\,)$} $ associated with the following quadratic form:
    $$\textit{$ q_{B}^{\lbrace T \rbrace} $}(u)=\int_{0}^{T}|u^{'}(\tau)|^{2}(1-B\tau) \mathrm{~d}\tau -|u(0)|^{2}\,\,\, .$$  with domain
$$\,\,\,\mathrm{Dom}(\mathcal{H}_{B}^{\lbrace T\rbrace})=\lbrace \,u \in  \textit{$H^{1}(0,T)$} \,:\, u(T)\,\,=0\, \rbrace\,.\,\,$$ The operator $\mathcal{H}_{B}^{\lbrace T \rbrace}$ is with a compact resolvent. Hence the spectrum $\sigma(\mathcal{H}_{B}^{\lbrace T \rbrace})\subseteq \mathbb{R}\,$ is purely discrete and consists of a strictly increasing sequence of eigenvalues
 $\Big(\lambda_{n}\Big(\mathcal{H}_{B}^{\lbrace T \rbrace}\Big)\Big)_{n\in \mathbb{N}^{*}}\,.$\\ 
 
The following proposition gives an asymptotic two-term expansion of the eigenvalue $\lambda_{1}(\mathcal{H}_{B}^{\lbrace T \rbrace})\,.$
\begin{proposition}
\label{P1}
$($Asymptotic of $\lambda_{1}(\mathcal{H}_{B}^{\lbrace T \rbrace})\,)\,\,$ There exist constant $C>0$ and $T_{0}>0$, such that for all $T\geq T_{0}$,  $B\in (-1/(3T),1/(3T))\,,$ we have:
 $$|\,\lambda_{1}(\mathcal{H}_{B}^{\lbrace T \rbrace})-(-1-B)\,|\leq CB^{2}.$$
\end{proposition}
After the change of function $u=(1-B\tau)^{-\frac{1}{2}}\tilde{u}\,,$ 
the new Hilbert space becomes  $\textit{$L^{2}(0,T)$}, d\tau)$, the form domain is always independent of the parameter and the expression of the operator depends on $B$ :
$$\mathcal{\tilde{H}}_{B}^{\lbrace T \rbrace}=-\dfrac{d^{2}}{d\tau^{2}}-\dfrac{B^{2}}{4(1-B\tau)^{2}}\, ,$$ 
with the new Robin condition at $ 0 $ denoted by $\tilde{u}^{'}(0)=\Big(-1-\dfrac{B}{2}\Big)\,\tilde{u}(0) \,\text{and}\,\, \tilde{u}(T)\,\,=0 \,.$\\
Note that the associated quadratic form is defined by
 $$\textit{$ \tilde{q}_{B}^{\lbrace T \rbrace} $}(\psi)=\int_{0}^{T}|\partial_{\tau}\psi|^{2}d\tau -\int_{0}^{T}\dfrac{B^{2}}{4(1-B\tau)^{2}}|\psi|^{2}d\tau -\Big(1+\dfrac{B}{2}\Big)|\psi(0)|^{2}\, .$$

For further use, we would like to estimate $\|\tau \tilde{u}_{B}^{\lbrace T \rbrace}\|_{\textit{$L^{2}(0,T)$}, d\tau)}$ and $\|\partial_{B}\tilde{u}_{B}^{\lbrace T \rbrace}\|_{\textit{$L^{2}(0,T)$}, d\tau)} $ uniformly with respect to $B$ and $T$.
\begin{proposition}
\label{P1}
There exist constants $C>0$ , $\alpha>0$ and $T_{0}>0$ such that, for all $T\geq T_{0}\,,$ $B\in (-1/(3T),1/(3T))\,,$ we have:
$$\|e^{\alpha \tau} \tilde{u}_{B}^{\lbrace T \rbrace}\|_{\textit{$L^{2}(0,T)$}, \mathrm{~d}\tau)}\leq C\,,$$ with  $\tilde{u}_{B}^{\lbrace T \rbrace}$  is the normalized eigenfunction associated with $\lambda_{1}(\mathcal{H}_{B}^{\lbrace T \rbrace})\,.$
\end{proposition}
\begin{proof}
Let $\phi$ a smooth function. By the formula of integration by parts, we get:
\begin{align*}
\Big\langle\,\mathcal{H}_{B}^{\lbrace T \rbrace} \tilde{u}_{B}^{\lbrace T \rbrace}, e^{2\phi}\tilde{u}_{B}^{\lbrace T \rbrace} \,\Big\rangle&=\int_{0}^{T}\Big|\partial_{\tau}\Big(e^{2\phi}\tilde{u}_{B}^{\lbrace T \rbrace} \Big)\Big|^{2}\mathrm{~d}\tau -\int_{0}^{T}\dfrac{B^{2}}{4(1-B\tau)^{2}}e^{2\phi}|\tilde{u}_{B}^{\lbrace T \rbrace} |^{2}\mathrm{~d}\tau \\&\,\,\,\,\,\,\,\,\,\,\,\,\,-\Big(1+\dfrac{B}{2}\Big)e^{2\phi(0)}|\tilde{u}_{B}^{\lbrace T \rbrace}(0)|^{2}-\|\phi^{'}e^{\phi}\,\tilde{u}_{B}^{\lbrace T \rbrace} \|^{2}\\&=\lambda_{1}(\mathcal{H}_{B}^{\lbrace T \rbrace})\,\|e^{\phi}\,\tilde{u}_{B}^{\lbrace T \rbrace} \|^{2}\,.
\end{align*}
We have $$|\tilde{u}_{B}^{\lbrace T \rbrace}(0)|^{2}=-2\int_{0}^{T} \partial_{\tau}\tilde{u}_{B}^{\lbrace T \rbrace}(\tau)\tilde{u}_{B}^{\lbrace T \rbrace}(\tau)\mathrm{~d}\tau\,,$$
according to a trace theory, there are constants $C>0$ such that, for all $\epsilon>0,$ we have:
\begin{equation}
\label{Q}
  |\tilde{u}_{B}^{\lbrace T \rbrace}(0)|^{2} \leq C \epsilon\, \|\partial_{\tau}\tilde{u}_{B}^{\lbrace T \rbrace}\|^{2}_{\textit{$L^{2}(0,T)$}} +C\epsilon^{-1} \|\tilde{u}_{B}^{\lbrace T \rbrace}\|^{2}_{\textit{$L^{2}(0,T)$}}   
  \end{equation}
 We replace \eqref{Q} in the form of $\textit{$ \tilde{q}_{B}^{\lbrace T \rbrace} $}$ and we use Proposition \ref{P1}, there exist constants $C>0$ such that:
  \begin{align*}
    \bigg(1-C\epsilon-\dfrac{C\epsilon B}{2}\bigg)\|\partial_{\tau}\tilde{u}_{B}^{\lbrace T \rbrace}\|^{2}_{\textit{$L^{2}(0,T)$}}\leq \bigg(\lambda_{1}(\mathcal{H}_{B}^{\lbrace T \rbrace}) + C\epsilon^{-1}\Big(1+\dfrac{B}{2}\Big)+C\bigg)\|\tilde{u}_{B}^{\lbrace T \rbrace}\|^{2}_{\textit{$L^{2}(0,T)$}}  \,
\end{align*}
  for $\epsilon=\dfrac{1}{3C}\,\,$ and $\,\,B\rightarrow 0\,$, we obtain:
$$\|\tilde{u}_{B}^{\lbrace T \rbrace}\|_{\textit{$H^{1}(0,T)$}}\leq C\,,$$
then, $$ |\tilde{u}_{B}^{\lbrace T \rbrace}(0)|^{2}\leq C\,. $$
Which implies the simple estimate
$$-\int_{0}^{T}\dfrac{B^{2}}{4(1-B\tau)^{2}}e^{2\phi}|\tilde{u}_{B}^{\lbrace T \rbrace} |^{2}\mathrm{~d}\tau -\Big(1+\dfrac{B}{2}\Big) C 
-\int_{0}^{T} |\phi^{'}|^{2}e^{2\phi}|\tilde{u}_{B}^{\lbrace T \rbrace} |^{2}\mathrm{~d}\tau
\leq\lambda_{1}(\mathcal{H}_{B}^{\lbrace T \rbrace})\,\|e^{\phi}\,\tilde{u}_{B}^{\lbrace T \rbrace} \|^{2}\,,$$ 
and thus,
$$\int_{0}^{T} \Big[-\dfrac{B^{2}}{4(1-B\tau)^{2}}-\lambda_{1}(\mathcal{H}_{B}^{\lbrace T \rbrace})- |\phi^{'}|^{2}\Big]   e^{2\phi} \,|\tilde{u}_{B}^{\lbrace T \rbrace} |^{2}\mathrm{~d}\tau \leq C \,.$$\\
As $T\rightarrow +\infty\,$, we have  $-1\leq B\leq1$, whence
\begin{enumerate}
\item[$\bullet$] $\dfrac{B^{2}}{4(1-B\tau)^{2}}\leq \dfrac{9}{16} B^{2}\,,$
\item[$\bullet$] $\lambda_{1}(\mathcal{H}_{B}^{\lbrace T \rbrace})\leq -1 -B\,.   $
\end{enumerate}
We choose $\phi=\alpha\, \tau$ with $\alpha$ a real positive constant.
For $\alpha<\dfrac{\sqrt{5}}{4}\,,$ we obtain:
$$-\dfrac{9}{16} B^{2}+1+B-\alpha^{2} \geq \dfrac{1}{16}\,,$$
then, $$\int_{0}^{T} e^{2\alpha\, \tau} \,|\tilde{u}_{B}^{\lbrace T \rbrace} |^{2}\,d\tau \leq C\,.$$
\end{proof}
\begin{lemma}
\label{L)}
There exist constants $C>0$ and $T_{0}>0$, such that, for all  $T\geq T_{0} \,$ and $\,B\in (-1/(3T),1/(3T))\,,$ we have:
$$\Big|\lambda_{1}(\mathcal{H}_{B}^{\lbrace T \rbrace})\Big|\leq C\,,   $$
$$\|\partial_{B}\tilde{u}_{B}^{\lbrace T \rbrace}\|_{\textit{$L^{2}(0,T)$}, \mathrm{~d}\tau)}\leq C\,.    $$
\end{lemma}
\end{appendix}
\section*{Acknowledgement}
I would like to express my deep gratitude to my supervisors, Nicolas Raymond and Ayman Kachmar,  for their patient guidance,  enthusiastic encouragement and useful criticism.  I would also like to thank Bernard Helffer for his attentive reading and  advice.

\end{document}